\documentclass[12pt,reqno]{amsart}
\usepackage{fullpage}
\usepackage{times}
\usepackage{amssymb, amsthm, url} 
\newtheorem{theorem}{Theorem}[section]
\newtheorem{prop}[theorem]{Proposition}
\newtheorem{lemma}[theorem]{Lemma}

\newtheorem{corollary}[theorem]{Corollary}
\newtheorem{claim}[theorem]{Claim}
\theoremstyle{definition}
\newtheorem{rem}[theorem]{Remark}

\usepackage[colorlinks=true, pdfstartview=FitH, linkcolor=blue, citecolor=blue, urlcolor=blue]{hyperref}
\usepackage[utf8]{inputenc}
\usepackage[english]{babel}
\usepackage{comment}
\usepackage{xcolor}
\newcommand{\eps}{\varepsilon}
\newcommand{\Z}{\mathbb{Z}}
\usepackage{mathtools}
\newenvironment{poc}{\begin{proof}[Proof of claim]}{\end{proof}}

\title{An inverse theorem on sets with rich additive structure modulo primes}

\author{Ernie Croot}
\address{School of Mathematics\\ Georgia Institute of Technology\\ Atlanta, GA 30332\\ United States}
\email{ernest.croot@math.gatech.edu}
\author{Junzhe Mao}
\address{School of Mathematics\\ Georgia Institute of Technology\\ Atlanta, GA 30332\\ United States}
\email{jmao87@gatech.edu}
\author{Chi Hoi Yip}
\address{School of Mathematics\\ Georgia Institute of Technology\\ Atlanta, GA 30332\\ United States}
\email{cyip30@gatech.edu}
\begin{document}

\begin{abstract}  In this paper, we prove several results on the structure of maximal sets $S \subseteq [N]$ such that $S$ mod $p$ is contained in a short arithmetic progression, or the union of short progressions, where $p$ ranges over a subset of primes in an interval $[y,2y]$ with $(\log N)^{O(1)} < y \leq N$. We also provide several constructions demonstrating the sharpness of our results. Furthermore, as an application, we provide several improvements on the larger sieve bound for $|S|$ when $S$ mod $p$ has strong additive structure, parallel to the work of Green--Harper and Shao for improvements on the large sieve.
\end{abstract} 

\maketitle

\section{Introduction}

Many problems in analytic number theory boil down to bounding the size of a set (both from above and from below) contained in a certain set of residue classes mod $p$ for various sets of primes $p$; and then sieve methods are the primary tools for doing so.  However, there are also many questions that require solving an {\it inverse} problem.  That is, if we let $S \subseteq [N]$ be a maximal set of integers in this interval where the residue classes mod $p$ occupied by $S$ have some particular pattern for many primes $p$, what can one say about the {\it structure} of the set $S$ beyond just its size? Rather than fully addressing this question, many results in the literature~\cite{green, helfgott, S15, shao, walsh1,walsh2} attempt to establish the following dichotomy:  
\medskip
\begin{center}
Either $|S|$ is {\it much smaller} than the upper bound given by various sieve methods such as the large/larger sieve, or else $S$ has some non-trivial structure, in particular {\it algebraic} structure.
\end{center}
\medskip
While some progress has been made, in most cases, such a dichotomy has not been confirmed; see, for example, the conjectures stated in \cite[Conjecture 1.3]{S15}
and \cite[Conjecture 1.7]{green}.

Before we can discuss results from the literature, as well as state our main theorems, we need to introduce the following notation:  if $A$ is a set of integers, then for an integer $m$, we let $A_m$ denote the image of $A$ under the canonical homomorphism $\varphi: {\mathbb Z} \to {\mathbb Z}_m$, where ${\mathbb Z}_m$ is shorthand for ${\mathbb Z}/m{\mathbb Z}$. Throughout the paper, $p$ denotes a prime. Given an integer $n$, a prime $p$, and a subset $R\subseteq \Z_p$, we write $n \in R \pmod p$ to mean $n \equiv x \pmod p$ for some $x\in R$. 

Green and Harper \cite{green} developed some fairly general results about what one can say about $S\subseteq[N]$ when it occupies less than about $\alpha p$ residue classes mod $p$ for certain sets of primes $p$ and for $0 < \alpha < 1$. For example, they \cite[Theorem 1.3]{green} showed that if $S_p$ is contained in a mod $p$ arithmetic progression $R_p$ of length $(1-\eps)p$ for all $p \leq N^{1/2}$, then the classical large sieve bound $|S|\ll N^{1/2}$ can be improved to $|S|\ll_\eps N^{1/2-\eps'}$, where $\eps'>0$ depends on $\eps$. Shao \cite[Theorem 1.5]{shao} also established a result of the same flavor. Furthermore, they \cite[Theorem 1.2]{green} demonstrated that the bound on $|S|$ can be further significantly improved when $S_p$ is contained in some interval of $\Z_p$ with length $(1-\eps)p$, for each prime $p \geq 2$. 

The Inverse Goldbach Problem \cite{elsholtz, EH, green, ostmann, shao} can also be understood to be a kind of inverse sieve problem, though it takes a bit of work to see how.  Elsholtz \cite{elsholtz}, for instance, managed to reduce the problem to understanding the possible sets $A, B \subseteq [N]$ where $A_p$ and $B_p$ are about size $p/2$ for many of the primes $p < N^{1/2-o(1)}$. Currently, we still have a limited understanding of the structure of such sets $A$ and $B$, and thus the inverse Goldbach Problem remains unsolved. We refer to some partial progress in Green--Harper~\cite[Theorem 1.4]{green}.

Another class of problems where the theorems we prove in this paper are relevant is understanding sets $S \subset [N]$, where, for an integer $0 \leq x \leq p-1$, $p$ prime, the counting function 
$$
C_p(x)\ :=\ \#\bigg\{S' \subseteq S\ :\ \sum_{s \in S'} s \equiv x \pmod{p}\bigg\}
$$
is non-uniform; that is, there exists $x$ such that $C_p(x)$ deviates substantially from $2^{|S|}/p$. Now, if this is to hold for all primes $p \in [y,2y]$ simultaneously, and $|S|>y^{1/2-\delta}$ for some $\delta > 0$, then what kind of structural properties does it force $S$ to have? In Section~\ref{oddsandends}, we will discuss how this essentially implies $S_p$ is ``mostly contained" in an arithmetic progression of length $p^{3/4+\delta}$ (also see \cite{desh, NV} for some similar results) for all primes $p\in [y,2y]$; and, furthermore, in the case of sets $S$ where $C_p(0)=0$ for all primes $p\in [y,2y]$ and $y>(\log N)^2$, we have the stronger conclusion that $S_p$ is ``mostly contained" in an arithmetic progression of size $p^{1/2+\delta+o(1)}$ for most primes $p\in [y,2y]$. Perhaps, one could deduce from the information on $S_p$ that the set $S$ is a long arithmetic progression, or at the very least {\it contains} a long arithmetic progression. Our results in Section~\ref{sec:maintheorem} are not {\it quite} able to determine the structure of $S$ in this case, due to the ``mostly contained" instead of ``completely contained" part; but it indicates a direction that one could go in, using a more general version of our results.  In fact, this is one of the motivations behind Theorem \ref{thm:kAP}, which is indeed a generalization of the type we thought might help.

A specific example in the above setting would be to suppose $S$ is a maximum subset of $\{1,...,N\}$ whose subset sums are not divisible by any prime $p \in [y,2y]$.  This would be the case $C_p(0) = 0$ for each prime $p \in [y,2y]$.  If one takes $S$ to be all integers in $\{d,2d,..., \lfloor(2y)^{1/2}-1\rfloor d\}$, 
where $d$ is not divisible by any of the primes in $[y,2y]$, then $S$ is certainly an arithmetic progression, like our theorems predict; and, furthermore, none of the subset sums of elements of $S$ are divisible by any prime in $[y,2y]$.  Perhaps these are, in fact, the largest sets $S$ for this problem. Note that this example also shows why the above assumption $|S| > y^{1/2-\delta}$ is reasonable.

There are also inverse sieve problems in higher dimensions where, instead of having $S \subseteq [N]$, one has $S \subseteq 
[N]^d$ for some $d \geq 2$.  For example, see the work of Helfgott--Venkatesh \cite{helfgott} and Walsh \cite{walsh1,walsh2}. In particular, Walsh \cite{walsh2} studied sets $S\subseteq [N]^d$ such that $S$ occupies $\ll p^{\kappa}$ residues classes for each prime $p$, where $\kappa$ is a real number with $0\leq \kappa<d$. He showed that a positive density subset of $S$ lies in the zero set of a polynomial with low complexity. 

In this paper, we focus on the same setting as Walsh~\cite{walsh2} for $d=1$ under the additional assumption that $S_p$'s have rich additive structure. In Section~\ref{sec:maintheorem}, we prove some inverse theorems when $S_p$'s are contained in short arithmetic progressions, and then prove results parallel to those of Green--Harper and Shao mentioned above. In Section~\ref{sec:extension}, we consider more general settings.

\subsection{Main results}\label{sec:maintheorem}
In general, unless the sets $S_p$ have a fairly restricted pattern, we might not expect to be able to say much about the structure of the set $S$.  However, in the case where the sets $S_p$ are short arithmetic progressions, we can prove a fairly strong structural result. Our first result is the following theorem, which one can think of as a kind of inverse sieve result:

\begin{theorem}\label{maintheorem}
There is an absolute constant $C>0$ such that the following holds. 
Let $\eps \in (0, 1/2)$.  Suppose $N$ is a positive integer with $N> N_0(\eps)$, and $y$ satisfies
\begin{equation*}
(4C\log N)^{1/2\eps}\ \leq\ y\ \leq\ N.
\end{equation*}
Let ${\mathcal P}$ be a subset of primes 
in $[y,2y]$ satisfying 
$$
|{\mathcal P}|\ \geq\ \frac{C y^{1-2\eps} \log N}{\log y},
$$
For each prime $p \in {\mathcal P}$, choose a prescribed arithmetic progression $R_p \subseteq {\mathbb Z}_p$ with $|R_p| \leq p^{1/2-\eps}+1$. Let $A$ be the set of all integers $n\in \{0,1,\ldots,N\}$ such that $n$ mod $p$ is in $R_p$ for every prime $p \in {\mathcal P}$.  Then, we must have that $A$ is an arithmetic progression of length at most $(2y)^{1/2-\eps}+1$. 
\end{theorem}

We note that there is not much lost here in restricting to primes in a dyadic interval $[y,2y]$, as opposed to an interval like $[y,10y]$ or even $[y, y^{1+\eps}]$ (adjusting the lower bound on $|{\mathcal P}|$ as needed),
since among such primes, once one has a subinterval of the type $[y', 2y']$ where enough primes satisfy the assumptions in this theorem, then we have that the corresponding set $A$ is already an arithmetic progression; Lemma \ref{lem2} then shows that including the remaining primes in the sieve, we would still have that the corresponding set $A$ is an arithmetic progression.

We have a series of remarks (together with proofs) concerning
different ways in which this theorem is near best possible and cannot be improved.  These are discussed in Section \ref{sec:extrarem}.  Roughly, these will
show that the length $(2y)^{1/2-\eps} + 1$ 
on $A$ is essentially the best possible
bound; the lower bound of $y \gg 
(\log N)^{1/2\eps}$ is a necessary assumption; the upper bound on $|R_p|$ is essentially necessary, in order to get the conclusion in the theorem; the lower bound on $|{\mathcal P}|$ is essentially necessary; and then we also give constructions of sets $R_p$ where $A$ is an arithmetic progression with endpoints nearly flush with the endpoints of $\{0,...,N\}$ (the endpoints being $0$ and $N$); finally, we show that it does not extend to polynomial progressions, in the sense that if the $R_p$ are the image of a short interval under a polynomial of degree $2$ and higher (different polynomials allowed for different choices of $p$), it does not follow that $S$ need have polynomial-like structure. 

Using results on the distribution of primes in short intervals, we have the following strengthening of Theorem~\ref{maintheorem} when $\mathcal{P}$ is the set of all primes in $[y,2y]$. In this case, the upper bound on $|A|$ is also asymptotically sharp; see Section~\ref{sec:2.1}.

\begin{corollary}\label{cor:allprimes}
There is an absolute constant $C_0$, such that under the assumptions of Theorem~\ref{maintheorem}, if moreover $\mathcal{P}$ is the set of all primes in $[y,2y]$ and $N>N_0(\eps)+C_0$, then $A$ is an arithmetic progression of length at most $(1+o(1))y^{1/2-\eps}$. 
\end{corollary}

We also have the following strengthening of Theorem~\ref{maintheorem} when each $R_p$ is a short interval, though the proof is much simpler. 

\begin{theorem}\label{thm:interval}
    Let $\eps\in (0,1)$. Suppose $N$ is a positive integer with $N>N_0(\eps)$, and $y$ satisfies 
    \[(16\log N)^{1/\eps}\leq y\leq N.\]
    Let $\mathcal{P}$ be a subset of primes in $[y,2y]$ satisfying
    \[|\mathcal{P}| \geq \frac{4(2y)^{1-\eps}\log N}{\log y}.\]
    For each prime $p\in \mathcal{P}$, choose a prescribed interval $I_p\subseteq \Z_p$ with $|I_p|\leq p^{1-\eps}+1$. Let $A$ be the set of all integers $n\in \{0,1,\ldots, N\}$ such that $n$ mod $p$ is in $I_p$ for every prime $p\in \mathcal{P}$. Then $A$ is an interval of length at most $p_0^{1-\eps}+1$, where $p_0 = \min\{p: p\in \mathcal{P}\}$.
\end{theorem}

The theorem is clearly optimal by considering the example $A=\{0,1,\ldots, \lfloor p_0^{1-\eps} \rfloor\}$ so that $A_p=\{0,1,\ldots, \lfloor p_0^{1-\eps} \rfloor\}$ for all $p\geq p_0$. Our motivation for this theorem was mainly to understand how better bounds than in Theorem \ref{maintheorem} are possible if the arithmetic progressions are replaced by intervals.  However, in Section~\ref{future} item (5), we relate Theorem \ref{thm:interval} to certain extensions of Kronecker's theorem; so there are other potential motivations for it. 

\medskip

Next, we discuss several generalizations and variants of Theorems \ref{maintheorem} and \ref{thm:interval}. 

The first generalization concerns what happens when $R_p$'s are progressions with length $|R_p|\gg p^{\theta}$ for some $\theta \in [1/2,1)$. In this case, we will see in Section~\ref{sec:2.3} that there are examples for which $|A|\gg y^{2\theta-1}$ while $A$ fails to be an arithmetic progression (more generally, we show that $A$ can fail to be a convex progression of bounded rank). This 
shows that Theorem~\ref{maintheorem} does not extend to longer progressions and Theorem~\ref{thm:interval} does not extend to general progressions, indicating a potentially more complicated structure for the set $A$. Nevertheless, the following theorem establishes that if locally $A$ ``correlates" with some ``possibly very short" arithmetic progression, then globally $A$ must be an arithmetic progression.

\begin{theorem}\label{thm:AP1/2}
    Let $\theta\in [1/2,1)$, $c\in [2\theta-1,\theta]$, and $\eps>0$. Suppose $N$ is a positive integer with $N> N_0(\eps)$, and $y$ satisfies
\begin{equation*}
(16\log N)^{1/2\eps}\ \leq\ y\ \leq\ N.
\end{equation*}
For each prime $p \in [y,2y]$, choose a prescribed arithmetic progression $R_p \subseteq {\mathbb Z}_p$ with $|R_p| \leq p^{\theta}$. Let $A$ be the set of all integers $n\in \{0,1,\ldots,N\}$ such that $n$ mod $p$ is in $R_p$ for every prime $p \in [y,2y]$.  If there exists an arithmetic progression $Q\subseteq \Z$ with $|Q|\geq y^c$ such that $$|A\cap Q|\geq 8y^{\frac{2\theta-1-c}{2}+\eps}|Q|,$$ then $A$ is an arithmetic progression.
\end{theorem}

Let $A$ be the set from the above theorem. Since $|A_p|\leq p^{\theta}+1$ for all primes $p\in [y,2y]$, Gallagher's larger sieve (Lemma~\ref{GS}) readily implies that $|A|\ll_{\theta} y^{\theta}$ (as mentioned in the remark after Theorem~\ref{thm:interval}, such a bound is sharp up to the implied constant). Building on Theorem~\ref{thm:AP1/2}, our next theorem shows that either we get an improved power-saving upper bound on $|A|$, or $A$ has to be an arithmetic progression.

\begin{theorem}\label{thm:largeA}
    Let $\theta\in [1/2,1)$ and $\eps>0$. Suppose $N$ is a positive integer with $N> N_0(\theta, \eps)$, and $y$ satisfies 
    \[(16\log N)^{1/\eps}\leq y\leq N.\] 
    For each prime $p \in [y,2y]$, choose a prescribed arithmetic progression $R_p \subseteq {\mathbb Z}_p$ with $|R_p| \leq p^{\theta}$. Let $A$ be the set of all integers $n\in \{0,1,\ldots,N\}$ such that $n$ mod $p$ is in $R_p$ for every prime $p \in [y,2y]$. Then the following statements hold:
    \begin{enumerate}
        \item If $|A|\geq 64\sqrt{2}y^{\theta-\frac{1-\theta}{2}+\eps},$ then $A$ is an arithmetic progression. 
    \item If $|A|\geq 80y^{2\theta-1+\eps}$, then $A$ is contained in an arithmetic progression of length at most $64y^{\theta}$.
    \end{enumerate}
    
\end{theorem}

In particular, in this setting, part (1) of the above theorem confirms the dichotomy ``either a set is significantly smaller compared to the sieve bound, or it possesses some strong algebraic structure" mentioned at the beginning of the paper. On the other hand, there are constructions such that $|A|\gg y^{2\theta-1}$ while $A$ is not contained in an arithmetic progression of length $y^{O(1)}$; see Section~\ref{sec:2.3}. Thus, part (2) of the above theorem is essentially sharp.

\medskip

The second generalization concerns what happens if $R_p$ is the {\it union} of several arithmetic progressions of length at most $p^{1/2-\varepsilon}$, rather than just a single progression.  One of the motivations and hopes for this theorem was that it could help solve the problem alluded to before, where if $C_p(x)$ deviates substantially from $2^{|S|}/p$, then it means $S_p$ is mostly contained in a short arithmetic progression mod $p$. The thought was that those extra pieces of $S$ that are not neatly contained in that short arithmetic progression could still, {\it themselves} be contained in the union of a small number of other short arithmetic progressions.  That may be the case, but we have not yet proved it.

Another perhaps {\it stronger} motivation of Theorem~\ref{thm:kAP} below is that it can maybe be used to prove that if $R_p$ is contained in a generalized arithmetic progression of low rank, 
then $S$ is a dense subset of a generalized arithmetic progression of low rank.  It might 
be used to prove this in the special case where $R_p$ is a generalized progression that is 
``skinny" along all but one direction, but is 
``long" in the remaining direction.  Basically, $R_p$ would be a short union of translates of arithmetic progressions on the ``long" directions.

Proving structural results when the $R_p$ is a union of several arithmetic progressions is significantly more challenging than the case where $R_p$ is just a single arithmetic progression; and to make it work, we need to make use of a kind of arithmetic regularity lemma, in addition to Theorem~\ref{maintheorem}.  This new theorem is as follows.

\begin{theorem}\label{thm:kAP}
Let $\eps \in (0, 1/2)$ and $k$ be a positive integer.  Suppose $N$ is a positive integer with $N> N_0(\eps)$, and $y$ satisfies
\begin{equation*}
(\log N)^{1/\eps}\ \leq\ y\ \leq\ N.
\end{equation*}
For each prime $p \in [y,2y]$, choose $k$ prescribed arithmetic progressions $R_p^{(1)}, R_p^{(2)}, \ldots, R_p^{(k)}$ in $\Z_p$ each with length at most $p^{1/2-\eps}$, and let $R_p=\cup_{i=1}^k R_p^{(i)}$. Let $A$ be the set of all integers $n\in \{0,1,\ldots,N\}$ such that $n\in R_p \pmod p$ for every prime $p \in [y,2y]$.  Then, $A$ is contained in the union of the $k$ arithmetic progressions, each of length at most $\exp(\exp(O(k\log k/\eps))) y^{1/2-\eps}$. 
\end{theorem}

In this case, it is impossible to obtain a strong inverse result like Theorem~\ref{maintheorem}. Indeed, $R_p$ may be a union of $k$ short arithmetic progressions for each prime $p\in [y,2y]$, while $A$ itself fails to be a union of $k$ arithmetic progressions, see Remark~\ref{rem:kAPno}.

\medskip
The third variant concerns bounding the size of one-dimensional ill-distributed sets. Following \cite{helfgott,walsh2}, we call a set $A\subseteq[N]$ to be \emph{ill-distributed} if for some $\theta\in (0,1)$, $|A_p|\leq p^\theta$ holds for {\it every} prime $p$. As explained by Walsh~\cite[Example 2.1]{walsh2}, it follows from Gallagher's larger sieve (Lemma~\ref{GS}) or \cite[Theorem 1.2]{walsh2} that $|A|\ll(\log N)^{\frac{\theta}{1-\theta}}$; however, it seems that the largest such $A$ one can possibly construct has size about $(\log N)^\theta$. We are not aware of any previous results improving the bound $|A|\ll(\log N)^{\frac{\theta}{1-\theta}}$, even when additional structural assumptions are made about $A_p$. Recall that in Theorems~\ref{maintheorem} and~\ref{thm:interval}, we only considered primes $p$ in a dyadic interval $[y,2y]$ and obtained asymptotically sharp upper bounds. With the help of Theorems~\ref{maintheorem} and~\ref{thm:interval}, we present two improved upper bounds on $|A|$ under the assumption that $A_p$ is contained in an interval (with $\theta<1$) or union of $k$ arithmetic progressions (with $\theta<1/2$) for a larger range of primes $p$. In particular, the next two theorems are parallel to the results of Green--Harper and Shao mentioned earlier. 

\begin{theorem}\label{thm:interval_sieve}
    Let $\eps\in (0,1)$. Suppose $N$ is a positive integer with $N>N_0(\eps)$ and $p_0 = p_0(N)$ is a prime depending on $N$ with $p_0\leq (16\log N)^{1/\eps}$.   
    For each prime $p_0\leq p\leq 2(16\log N)^{1/\eps}$, choose a prescribed interval $I_p\subseteq\Z_p$ with $|I_p|\leq p^{1-\eps}+1$. Let $A$ be the set of all integers $n\in \{0,1,\ldots,N\}$ such that $n \in I_p \pmod p$ holds for each $p_0\leq p\leq 2(16\log N)^{1/\eps}$. Then $A$ is contained in an interval of length at most $\max(p_0^{1-\eps}+1,N_0(\eps))$. 
\end{theorem}

\begin{theorem}\label{thm:kAP_sieve}
    Let $\eps\in (0,1/2)$, $k$ be a positive integer. Suppose $N$ is a positive integer with $N>N_0(\eps)$ and $p_0=p_0(N)$ is a function of $N$ with $p_0\leq (\log N)^{1/\eps}/2$. For each prime $p_0\leq p\leq (\log N)^{1/\eps}$, choose $k$ prescribed arithmetic progressions $R_p^{(1)}, R_p^{(2)}, \ldots, R_p^{(k)}$ in $\Z_p$ each with length at most $p^{1/2-\eps}$, and let $R_p=\cup_{i=1}^k R_p^{(i)}$. Let $A$ be the set of all integers $n\in \{0,1,\ldots,N\}$ such that $n \in R_p \pmod p$ holds for every prime $p_0\leq p\leq (\log N)^{1/\eps}$. Then $|A|\ll_{k,\eps}(\log N)^{1/2-\eps}+p_0^{1/2-\eps}$.
\end{theorem}

It is easy to show that Theorem~\ref{thm:interval_sieve} is sharp and Theorem~\ref{thm:kAP_sieve} is optimal up to the implied constant; see Remark~\ref{rem:sieve}.


\subsection{Other extensions}\label{sec:extension}
Besides just extending the length of the progression $R_p$ as in Theorems~\ref{thm:AP1/2} and~\ref{thm:largeA} or replacing it with a union of arithmetic progressions as in Theorem \ref{thm:kAP}, one could also even forego working with arithmetic progressions at all, and instead work with sets of high additive energy (generalized arithmetic progressions have this property). In this setting, we prove some inverse theorems of a similar flavor, although they are weaker than the ones discussed in Section~\ref{sec:maintheorem}.

The theorem below shows that if the additive energy of $S_p$ is large for all primes $p\in [y,2y]$, then the additive energy of $S$ has to be large. Recall that for a subset $S$ of an abelian group, its additive energy is defined to be $$E(S)=\#\{(a,b,c,d) \in S^4:a+b=c+d\}.$$

\begin{theorem}\label{thm:energy}
    Let $\theta\in(0,1)$ and $\delta>0$. Suppose $y$ satisfies 
    \[(\log N)^{2/(1-\theta)}<y\leq N,\]
    and $S\subseteq[N]$ is a set such that $|S_p|\leq p^{\theta}$ and $E(S_p)\geq \delta |S_p|^3$ hold for each prime $p\in[y,2y]$. Then $E(S)\gg_{\theta}\delta|S|^3$.
\end{theorem}

One can compare this with several results of Green and Harper~\cite[Section 3]{green}. In particular, \cite[Lemma 3.2]{green} implies that if $\delta>0$, and $S\subseteq [N]$ is a set such that $|S_p|\leq (p+1)/2$ and $E(S_p)\geq (\frac{1}{16}+\delta)p^3$ for each prime $p\leq N^{1/2}$, then $E(S)\gg \frac{\delta|S|^4}{N^{1/2}\log N}$. 

Next, we prove some inverse theorems for generalized arithmetic progressions, roughly speaking, if $S_p$ is a GAP with bounded rank for each prime $p\in [y,2y]$, then a positive proportion of $S$ must be contained in a GAP with bounded rank and volume. To state these theorems formally, we first recall some basic terminology on generalized arithmetic progressions. For an abelian group $G$, a \emph{generalized arithmetic progression (GAP) over $G$ of rank $r$} is of the form 
$$
\bigg\{a+\sum_{i=1}^{r} n_i v_i: 0\leq n_i\leq N_i \text{ for all } 1\leq i \leq r\bigg \},
$$
where $a,v_1, \ldots, v_r \in G$, and $N_1, \ldots, N_r$ are integers. More generally, a set $P \subset G$ is called a \emph{convex progression of rank $r$} if there is 
a symmetric convex body $Q \subset \mathbb{R}^r$, a homomorphism $\psi:\Z^r \to G$, and some element $a\in G$ such that $a+\psi(\Z^r \cap Q)=P$. Note that GAPs are special convex progressions. Our proof techniques for the next three results extend to the setting where $R_p$'s are convex progressions with bounded rank. For simplicity, we only state and prove the version for GAPs.

\begin{theorem}\label{thm:inverseGAP}
    Let $r$ be a positive integer, $\theta\in(0,1)$, and $\eps\in (0,1-\theta)$. Suppose $N$ is a positive integer with $N>N_0(r,\eps)$, and the two parameters $\delta$ and $y$ satisfy
    \[\delta\geq y^{\theta-1+\eps},\quad(\log N)^{2/\eps}<y\leq N.\]
    Suppose that $S\subseteq[N]$ is a set such that for each $p\in [y,2y],$ there is a GAP $R_p \subset \Z_p$ of rank at most $r$ with $|R_p|\leq p^\theta$, $S_p\subseteq R_p$, and $|S_p|\geq \delta |R_p|$. Then there exists a GAP $Q$ of rank $O( (r+\log\delta^{-1})^{1+o(1)})$ such that $S\subseteq Q$ and $|Q|\leq \exp(O((r+\log\delta^{-1})^{1+o(1)}))|S|$.
\end{theorem}
Note that the rank of $Q$ in Theorem~\ref{thm:inverseGAP} is potentially much larger than $r$. However, intuitively, $S$ should be efficiently contained in a GAP $Q$ of rank at most $r$. The following theorem partially justifies this, showing that we can cover $S$ using translates of a GAP $Q$ with rank at most $r$ rather than a single GAP. However, the number of translations here is approximately double-exponential in $r$, leading to a poorer bound on the density of $S$ in a single GAP compared to Theorem~\ref{thm:inverseGAP}.

\begin{theorem}\label{thm:inverseGAPsamerank}
    Let $r$ be a positive integer, $\theta\in (0,1)$, $\eta \in (0,1/2)$, and $\eps\in (0,(1-\theta)/r)$. Suppose $N$ is a positive integer with $N> N_0(r,\eps)$, and the two parameters $\delta$ and $y$ satisfy
    \[\delta\geq y^{\frac{(\theta-1)(1-\eta)}{r+1-\eta}+\eps},\quad\quad \left(8^{r+3}\log N\right)^{1/(r\eps)}<y\leq N.\]
    Let $\mathcal{P}$ be a subset of primes in $[y,2y]$ such that $|\mathcal{P}|\geq y/(4\log y)$. Suppose that $S\subseteq[N]$ is a set such that for each prime $p\in \mathcal{P},$ there is a GAP $R_p \subset \Z_p$ of rank at most $r$ such that $$|R_p|\leq p^\theta, \quad S_p\subseteq R_p, \quad |S_p|\geq \delta |R_p|.$$ Then there exists a GAP $Q$ of rank at most $r$ and size at most $2^{3r+1}\delta^{-(1+\frac{r}{1-\eta})}|S|$, such that $S$ is covered by at most $\exp(C\exp(Cr)\cdot\log\eta^{-1})$ translates of $Q$, where $C$ is an absolute constant.
\end{theorem}

As an application of Theorem~\ref{thm:inverseGAPsamerank}, we prove the following theorem, which can be viewed as a partial extension of Theorem~\ref{thm:largeA}. 

\begin{theorem}\label{thm:largeGAP}
 Let $r$ be a positive integer, $\theta\in (0,1)$, and $\eps\in (0,(1-\theta)/(r+1))$. Suppose $N$ is a positive integer with $N> N_0(r,\eps)$, and the two parameters $\delta$ and $y$ satisfy
    \[\left(8^{r+3}\log N\right)^{2/(r\eps)}<y\leq N.\]
Suppose that $S\subseteq[N]$ is a set such that for each prime $p\in [y,2y],$ there is a GAP $R_p \subset \Z_p$ of rank at most $r$ with $S_p\subseteq R_p$ and $|R_p|\leq p^\theta$. If $|S|\geq y^{\theta-\frac{1-\theta}{r+1}+\eps}$, then there exists a GAP $Q$ of rank $\leq r$ and size at most $2^{3r+5} y$, such that $S$ is covered by at most $\exp(C\exp(Cr)\cdot\log\eps^{-1})$ translates of $Q$, where $C$ is an absolute constant.
\end{theorem}

Observe that when $\theta<\frac{1-\theta}{r+1}$ (that is, $\theta<\frac{1}{r+2}$) and $\eps$ is sufficiently small, the condition $|S|\geq y^{\theta-\frac{1-\theta}{r+1}+\eps}$ holds trivially, unless $S$ is empty. Thus, in this setting, Theorem~\ref{thm:largeGAP} implies that if $y\geq (\log N)^{O_{r,\theta}(1)}$, then $S$ is always covered by $\ll_{r,\theta} 1$ many translates of a GAP $Q$ of rank $\leq r$ and size $\ll_{r} y$; this can be viewed as a partial extension of Theorem~\ref{maintheorem}.

\subsection{Future Directions and Questions} \label{future}
Here, we discuss a few problems worth considering as potential future directions. 
\begin{enumerate}

\item The first thing worth asking is whether the conclusion in Theorems \ref{thm:inverseGAP} and \ref{thm:inverseGAPsamerank} can be simultaneously enhanced so that $S$ is efficiently contained in a GAP of rank at most $r$.  Perhaps there are constructions showing that this is not always possible.

\item Suppose the answer to the above question is positive.  This is somewhat counterintuitive, since one would expect the rank of the GAP containing $S$ should {\it expand} as one adds more and more primes to the sieving process.  Could there perhaps be interesting uses of this phenomenon?  For example, perhaps certain search algorithms can be improved.  Imagine, for example, a search process where at each step we pass to a subset of previous numbers that mod $p$ are in a given arithmetic progression, then eventually search over some residual set.  If the arithmetic complexity of that residual set is low (e.g., the GAP has low rank), then it should make it easier to search through it.

\item We have a strong result for sets $S$ when $R_p$ has length at most $p^{1/2-\varepsilon}$, showing that $S$ is an arithmetic progression; and then we know that this does not continue to hold once $R_p$ is about $p^{1/2}$ (though, if we assume a little extra structure, as is done in Theorem \ref{thm:AP1/2}, we {\it can} deduce that $S$ is an arithmetic progression even when the $|R_p| > p^{1/2}$). However, it may still be true that $S$ has some very precise structure. Perhaps it turns out to be the case that 
$S$ is approximately a convex progression of rank at most $r-1$ when $|R_p| < p^{1-1/r-o(1)}$; see Section~\ref{sec:2.3} for further discussions. What is the true structure in these cases?

\item We have seen that it is not possible to prove that if the $R_p$ are polynomial images of intervals for $p\in [y,2y]$, then $S$, itself, is the polynomial image of intervals.  But perhaps this is still true when the $R_p$ are special types of polynomial images, or perhaps there is a yet more general structural result that covers this case.  It would be worth investigating this.

\item Theorem \ref{thm:interval} implies that for each $p \in {\mathcal P}$, if $\alpha_p \in {\mathbb R}$ and if we define $\lambda_p := 1/p$, then the set of integers $t \leq N$ satisfying 
$\| \lambda_p t - \alpha_p\| \ll p^{-\eps}$, if they exist at all, form a short interval of width at most $O(y^{1-\eps})$.  This formulation is related to what Gonek and Montgomery \cite{gonek} call a ``localized" and ``quantitative" version of Kronecker's Theorem, except that in our case the $\lambda_p$'s are specific numbers (reciprocals of primes), and the size of the error $O(p^{-\eps})$ is significantly smaller.  Our choosing the $\lambda_p$ to be reciprocals of primes from $[y,2y]$ ensures that linear forms $\mu \cdot \lambda \neq 0$ for $\lambda = (\lambda_p)_{p \in {\mathcal P}}$ and $\mu = (\mu_p)_{p \in {\mathcal P}}$ a non-zero vector of integers with $|\mu_p| < y$.  In fact, this is one of the simplest choices for the $\lambda_p$'s with this property.  Our conclusion does not claim that such $t$'s always exist, but merely points to what they could possibly be, which is an interval of width at most $O(y^{1-\eps})$, which is related to what Gonek and Montgomery call ``localized".\footnote{Gonek and Montgomery refer to a variant of Kronecker as ``localized" if for any interval $T$ such that $|T|$ satisfies certain constraints, there exists $t \in T$ so that $t$ so that $\|\lambda_i t - \alpha_i\| < \eps$ for all $i$ in the index set for the $\lambda_i$'s and $\alpha_i$'s.  What we are saying is related, but ultimately different.  We are saying that if there exists an interval containing such $t$, then in fact there is a whole interval $T$ where this is true.} The main question that all this motivates is:  can one generalize Theorem \ref{thm:interval} to make it more closely resemble -- while being much stronger than -- these generalizations of Kronecker's Theorem from the literature?  In particular, can one get a result where the $\lambda_p$'s are not required to be reciprocals of primes, but can include a broad range of sets of numbers with prime-like behavior (e.g., satisfy something close to the Fundamental Theorem of Arithmetic)?

\end{enumerate}

\subsection{Some odds and ends}\label{oddsandends}

Let $S$ be a set of integers with $|S| > y^{1/2-\delta}$ and let $p\in [y,2y]$. To understand the relation with inverse-sieves, define the Riesz product 
$$
R(t)\ :=\ \prod_{s \in S} \left ( 1 + e^{2\pi i s t} \right ),
$$
and then note that 
$$
C_p(x)\ =\ \frac{1}{p} \sum_{a=0}^{p-1} e^{-2\pi i a x/p} R(a/p).
$$
Using the identity 
$$
1 + e^{2\pi i u}\ =\ 2 \cdot e^{\pi i u}\cdot \bigg(\frac{e^{-\pi i u} + e^{\pi i u}}{2}\bigg) =\ 2 e^{\pi i u} \cos(\pi u),
$$
and using the fact that $|\cos(\pi u)|\ =\ 
\cos(\pi \|u\|)$, we have that
$$
|R(t)|\ =\ 2^{|S|} \prod_{s \in S} \cos (\pi \|st\|).
$$

If for all $a \in \{1,\ldots,p-1\}$, the value of $|R(a/p)|$ were much smaller than $2^{|S|}/p$, then all the $C_p(x)$ would have size about $2^{|S|}/p$. So if this is badly wrong for some $x$, it means there exists some $a \in \{1,\ldots,p-1\}$ where $|R(a/p)| \gg 2^{|S|}/p$, say.  Let $T^{(p)}\subseteq S$ be all those
$s \in S$ where $\|sa/p\| > |S|^{-1/2+\delta}$. 
Then, since $\cos(x) \leq 1 - x^2/4$ for $x$ near to $0$, from the above we would have
$$
|R(a/p)|\ \leq\ 2^{|S|} \prod_{s \in T^{(p)}} 
\cos(2\pi \|sa/p\|)\ \leq\ 2^{|S|} \left ( 1 - 
{\pi^2 \over |S|^{1-2 \delta}}\right )^{|T^{(p)}|}
\ \leq\ 2^{|S|} \exp\left ( - {\pi^2 |T^{(p)}| \over 
|S|^{1-2\delta}} \right ).
$$
Since $|R(a/p)| \gg \frac{2^{|S|}}{p}$, we have
$
|T^{(p)}|\ \ll\ |S|^{1-2 \delta} \log p
$
and in particular $|T^{(p)}|=o(|S|)$ since $|S|\geq y^{1/2-\delta}$. It follows that almost all $s \in S$ such that
$\|sa/p\| < |S|^{-1/2+\delta} < y^{-1/4 +\delta}$. Equivalently, $S_p$ is contained in the union of an arithmetic progression of length $\ll p^{3/4 +\delta}$ and an exceptional set of size $o(|S|)$. 

Next, assume additionally that $C_p(0)=0$ for all primes $p\in [y,2y]$ and $y>(\log N)^2$. We show that for most primes $p\in [y,2y]$,
all but $o(|S_p|)$ of the elements of 
$S_p$ lie in an arithmetic progression of length at most $O(p^{1/2+\delta+o(1)})$. By \cite[Corollary 2.3]{NV}, for each prime $p\in [y,2y]$, 
there is $b \in \{1,2,...,p-1\}$ and a subset $S'_p$ of $S_p$ with $|S'_p| = |S_p| - O(p^{6/13}\log^2 p)$, such that 
\begin{equation}\label{sumvu}
\sum_{a \in b \cdot S'_p} a\ <\ p
\end{equation}
(here, $b\cdot S'_p$ is viewed as a subset of 
$\{0,1,2...,p-1\}$).  If we knew that $|S_p| > p^{1/2-\delta-o(1)}$, say, then it would follow from inequality~\eqref{sumvu} that the number of $s\in S$ so that $b\cdot s$ mod $p$ is $> p^{1/2+\eps+\delta}$ is at most $p^{1/2-\eps-\delta}$.  So, for $0 < \eps+\delta < 1/26$, we have that all but $|S_p|p^{-\eps+o(1)}$ of the elements of $S_p$ lie in a progression of length at most
$O(p^{1/2+\eps+\delta})$. To finish up, a standard application of Gallagher's sieve (Lemma~\ref{GS}) shows that $|S_p| > p^{1/2-\eps-o(1)}$ for most primes $p\in [y,2y]$, as required. 
\medskip

\textbf{Notation.} We follow standard notations in analytic number theory. In this paper,~$p$ always denotes a prime, and $\sum_p$ and $\prod_p$ represent sums and products over all primes. We also use the Vinogradov notation $\ll$; we write $X \ll Y$ if there is an absolute constant $C>0$ so that $|X| \leq CY$.

\textbf{Organization of the paper.} In Section~\ref{sec:prelim}, we list some useful tools from arithmetic combinatorics. In Section~\ref{sec:extrarem}, we give a few remarks related to the sharpness of Theorem~\ref{maintheorem}. We then present the proof of Theorem~\ref{maintheorem}, and use similar ideas to prove Corollary~\ref{cor:allprimes} and Theorem~\ref{thm:interval} in Section~\ref{sec:theoremproofsection}. In Section~\ref{sec:1/2}, we prove Theorems~\ref{thm:AP1/2} and~\ref{thm:largeA} when $R_p$'s are longer progressions. In Section~\ref{sec:app}, we discuss the proof of Theorem~\ref{thm:kAP} when $R_p$'s are the union of several arithmetic progressions. We also present the proof of Theorems~\ref{thm:interval_sieve} and~\ref{thm:kAP_sieve} that improve the larger sieve. Finally, in Section~\ref{sec:GAP}, we prove Theorems~\ref{thm:energy},~\ref{thm:inverseGAP},~\ref{thm:inverseGAPsamerank}, and~\ref{thm:largeGAP}
related to inverse theorems on additive energy and GAPs. 

\section{Preliminaries}\label{sec:prelim}

In this section, we recall a few useful tools from arithmetic combinatorics. 

We begin with Gallagher's larger sieve \cite{gallagher} that has been mentioned several times in the introduction. We will use this repeatedly throughout the paper.

\begin{lemma}[Gallagher's larger sieve]\label{GS}  
Let $N$ be a positive integer and $A\subseteq\{1,2,\ldots, N\}$. Let ${\mathcal P}$ be a set of primes. For each prime $p \in {\mathcal P}$, let $A_p=A \pmod{p}$. For any $1<Q\leq N$, we have
$$
 |A|\leq \frac{\underset{p\leq Q, \ p\in \mathcal{P}}\sum\log p - \log N}{\underset{p\leq Q, \ p \in \mathcal{P}}\sum\frac{\log p}{|A_p|}-\log N},
$$
provided that the denominator is positive.
\end{lemma}

Given an abelian group $G$, a positive integer $h$, and some $A\subseteq G$, the \emph{$h$-fold sumset of $A$} is defined as \[hA = \{a_1+\cdots+a_h: a_i\in A\ \text{for all}\ 1\leq i\leq h\}.\] Next, we list the following two well-known results of Lev on $h$-fold sumsets. 

\begin{lemma}[{Lev~\cite[Theorem 2']{lev}}]\label{lem:lev1}
Let $B \subset \{1,2,\ldots, M\}$ be an arbitrary set of $n$ integers, and assume
that a positive integer $\kappa$ satisfies that $M\leq (\kappa+1)(n-2)+1$. Then there exist
positive integers $d,h$ such that $d\leq \kappa$, $h\leq 2\kappa+1$, and $hA$ contains $M$ consecutive multiples of $d$. 
\end{lemma}

\begin{lemma}[{Lev~\cite[Corollary 1]{lev2}}]\label{lem:lev2}
Let $B \subset \{0, 1,2,\ldots, M\}$ be a set of $n\geq 2$ integers such that $0,M\in B$ and $\gcd(B)=1$. Let $k$ be a positive integer such that 
$k(n-2)+1\leq M\leq (k+1)(n-2)+1$. Then for each positive integer $h\leq k$, we have
$|hB|\geq \frac{h(h+1)}{2}(n-2)+h-1$.
\end{lemma}

The following theorem is a version of the Freiman--Bilu theorem \cite{Bilu}, due to Green and Tao \cite{GT06}.
\begin{theorem}\label{thm:FB}
Let $A$ be a finite subset of a torsion-free abelian group with $|A+A|\leq K|A|$ and $|A|>2$. Then for any $\eps \in (0,1]$, one can cover $A$ by at most $\exp(CK^3\log^3 K)/\eps^{CK}$ translates of a GAP $Q$ with rank at most $\lfloor \log_2 K+\eps\rfloor$ and size at most $|A|$, where $C$ is an absolute constant.
\end{theorem}

Lastly, we recall a useful structural result for sets with relative polynomial growth, due to Sanders \cite{S13}. Here we say that a set $A\subseteq\Z$ has \emph{relative polynomial growth of order $d$} if 
\[|nA|\leq n^d|A|\quad\text{for all integer }n\geq 1.\]

\begin{theorem}[{\cite[Theorem 2.7]{S13}}]\label{thm:polygrowth}
    Suppose that $A\subseteq\Z$ has relative polynomial growth of order $d$. Then there is a centered convex progression $Q$ of rank at most $O(d\log^2 d)$, such that  \[A-A\subseteq Q\quad \text{and} \quad |Q|\leq \exp(O(d\log^2 d))|A|.\]
\end{theorem}

It is known that a centered convex progression can be efficiently contained in some GAP of the same rank. In particular, combining Theorem~\ref{thm:polygrowth} with the discrete John's theorem \cite[Lemma 3.36]{tao}, we have the following corollary.

\begin{corollary}\label{cor:polygrowth}
Suppose that $A\subseteq\Z$ has relative polynomial growth of order $d$. Then there is a GAP $Q$ of rank at most $O(d\log^2 d)$, such that  \[A-A\subseteq Q\quad \text{and} \quad |Q|\leq \exp(O(d\log^2 d))|A|.\]
\end{corollary}

The following lemma is helpful for showing that a set has relative polynomial growth of small order.

\begin{lemma}[{\cite[Corollary 5.3]{S13}}]\label{lem:polygrowth}
    Suppose that $X\subseteq\Z$ is a symmetric neighborhood (that is, $0\in X$, and $-X=X$). If $h$ is a positive integer such that $|(3h+1)X|<2^h |X|$, then $X$
    has relative polynomial growth of order $O(h)$. 
\end{lemma}

\section{Six remarks and results concerning Theorem \ref{maintheorem}}\label{sec:extrarem}

In this section, we discuss the sharpness of Theorem~\ref{maintheorem} from several different perspectives.

\subsection{The length $(2y)^{1/2-\eps} + 1$ 
on $A$ is best-possible}\label{sec:2.1}\

In Theorem~\ref{maintheorem}, $\mathcal{P}$ could be a set of primes with zero relative density; in particular, $\mathcal{P}$ could consist of primes that are very close to $2y$. Thus, to show that the bound $(2y)^{1/2-\eps}+1$ on $|A|$ is asymptotically sharp for Theorem~\ref{maintheorem}, it suffices to do so for Corollary~\ref{cor:allprimes}.

Consider the following example.  Let $N$ be sufficiently large, $\mathcal{P}$ consisting of all primes in $[y,2y]$, and taking $R_p=[0,\lfloor p^{1/2-\eps}\rfloor] \cap \Z_p$ for all the primes $p \in [y,2y]$. In this case, it is easy to verify directly that $A=\Z \cap [0, q^{1/2-\eps}]$, where $q$ is the smallest prime in $[y,2y]$. Alternatively, this follows from  Theorem~\ref{thm:interval}. By a theorem on the distribution of primes in short intervals due to Baker, Harman, and Pintz \cite{baker}, we have $q = y + O(y^{0.525})$ and thus $q^{1/2-\eps} = (1+o(1))y^{1/2-\eps}$. This shows that Corollary~\ref{cor:allprimes} is asymptotically sharp. 


\subsection{The lower bound $y \gg (\log N)^{1/2\eps}$ is a necessary condition}\

By the prime number theorem, we have
$\frac{y}{\log y}\gg |\mathcal{P}|\gg \frac{y^{1-2\eps}\log N}{\log y}$. Thus, in Theorem~\ref{maintheorem}, it is necessary to assume that $y\gg (\log N)^{1/2\eps}$ so that the theorem is non-trivial.

We remark that even if $\mathcal{P}$ consists of all primes in $[y,2y]$, it is still necessary to impose some lower bound on $y$ for Theorem~\ref{maintheorem} to hold. For example, if $c<1$ and $N$ is sufficiently large, then the theorem fails when $y<c\log N$ (and $y$ is sufficiently large). Indeed, if we take $R_p=[0,\lfloor p^{1/2-\eps}\rfloor]\cap \Z_p$ for all the primes $p \in [y,2y]$, then $A$ contains all integers in $[0,N]$ that are in the interval $[0, \lfloor q^{1/2-\eps} \rfloor]$ modulo $\prod_{y\leq p \leq 2p} p \asymp e^y<N^c$, where $q$ is the smallest prime in $[y,2y]$; in particular, $|A|\gg N^{1-c}y^{1/2-\eps}$ and $A$ is not an arithmetic progression.

\subsection{The upper-bound on $|R_p|$ of size $p^{1/2-\eps}$ is essentially necessary}\label{sec:2.3}\

Suppose in Theorem~\ref{maintheorem}, we replaced the upper bound $p^{1/2-\eps}+1$ on $|R_p|$ with $p^{\theta}$ for some fixed $\theta \in (0,1)$. Theorem~\ref{maintheorem} shows that $A$ is an arithmetic progression when $\theta<\frac{1}{2}$. In the example below, we show that for any fixed positive integer $r$, if $\theta> \frac{r}{r+1}$, then it is possible that $|A|\gg y^{(r+1)\theta-r}$ while $A$ is not a convex progression of rank at most $r$. In particular, when $\theta> \frac{1}{2}$, it is possible that $|A|\gg y^{2\theta-1}$ while $A$ is not an arithmetic progression; moreover, $A$ is not contained in an arithmetic progression of length $\ll y^\theta$.

Let $y=(\log N)^{C_r}$ for some sufficiently large constant $C_r$. For each $1\leq i \leq r+1$, choose a prime $q_i \in [y^{ir}, 2y^{ir}]$. Let $\mathcal{P}$ be the set of all primes in $[y,2y]$. By the pigeonhole principle,
for each prime $p \in \mathcal{P}$, we can can find an integer
$1 \leq d_p \leq p-1$ so that we have 
$$
\left \|{q_i d_p \over p} \right \|<\ {1 \over p^{1/(r+1)}}
$$
for all $1\leq i\leq r+1$. Choose such a $d_p$ and let $e_p\in \{1,2,\ldots, p-1\}$ such that $d_p e_p \equiv 1 \pmod d$, and define $R_p=\{je_p: |j|\leq p^{\theta}/2r\}$. In this case, by construction, $A$ must contain a proper GAP $Q$ of rank $r+1$, where
\[
Q = \left\{\sum_{i=1}^{r+1}k_iq_i: 0\leq k_i\leq \frac{1}{r!}y^{\theta-\frac{r}{r
+1}}\right\}.
\]
In particular, we must have $|A|\geq |Q|\gg y^{(r+1)\theta-r}$. 

Next, we show that $A$ is not a convex progression of rank $r$. Suppose otherwise that $A$ is a convex progression of rank $r$. Then we have $Q\subseteq A\subseteq A-A$ since $0\in A$. Since $|(A-A)_p|\leq |R_p-R_p|\leq 2|R_p|$, it follows that $|A-A|\ll y^{\theta}$ by Lemma~\ref{GS}. By the discrete John's theorem~\cite[Lemma 3.36]{tao}, there is a symmetric GAP $T$ of rank at most $r$ such that $A-A\subseteq T$ and $|T|\ll |A-A|\ll y^\theta$. Assume that $T = \{\sum_{i=1}^rm_it_i:-M_i\leq m_i\leq M_i\}$ for some positive integers $M_1,\ldots,M_r$. From $|T|\ll y^\theta$, we must have $M_i\ll y^\theta$ for all $1\leq i\leq r$. Now for each $1\leq j\leq r+1$, we can write $q_j$ as an integral linear combination of $t_i$'s:
\[q_j = \sum_{i=1}^r\alpha_{ji}t_i, \]
where $|\alpha_{ji}|\leq M_i\ll y^\theta$. By Siegel's lemma (see for example \cite[Lemma D.4.1]{HS00}), the system of $r$ linear equations
\[\sum_{j=1}^{r+1}\alpha_{ji}x_j = 0\quad\text{for }1\leq i\leq r\]
in $r+1$ unknowns $(x_1,\ldots,x_{r+1})$ has a nontrivial integer solution $(x_1,\ldots,x_{r+1})$ with $|x_j|\ll y^{\theta r}$ for all $1\leq j\leq r+1$. It follows that 
\[\sum_{j=1}^{r+1}x_jq_j = 0.\]
However, if $x_{r+1}\neq 0$, then $y^{r+r^2}\leq |q_{r+1}|\leq \sum_{j=1}^r|x_jq_j|\ll y^{\theta r+r^2}$, which is a contradiction when $y$ is large. Hence $x_{r+1}=0$. By induction, one can easily deduce that $x_j = 0$ for all $1\leq j\leq r+1$. This contradicts the fact that $(x_1,\ldots,x_{r+1})$ is nontrivial. Therefore, $A$ cannot be a convex progression of rank $r$.

Finally, observe that, when $r=1$, $A$ is not contained in any arithmetic progression of length $y^{O(1)}$. To see this, suppose there is some arithmetic progression $P\supseteq A$ with $|P|\leq y^{O(1)}$. Then $0,q_1,q_2\in P$ and thus $P$ has to be an interval. Thus, $|P|\geq q_2$. However, we can slightly modify the above construction by choosing $q_2$ to be close to $y^C$ for arbitrarily large $C$.

\subsection{Theorem \ref{maintheorem} is false if we restrict to certain thinner subsets of the primes}\label{extraremark1}\

Note that when $y=N$, Theorem \ref{maintheorem} requires that $|\mathcal{P}| \gg N^{1-2\eps}$. 
If we were to replace $\mathcal{P}$ with a thinner subset of primes with size $cN^ {1-2\eps}/\log N$, where $c>0$ is sufficiently small, then the conclusion is false by considering the following construction.  

First, we show that there exists some $x \in [N^{-1/2-\eps}/20, N^{-1/2-\eps}/10]$, such that the interval $I := [x, x + N^{-1-2\eps}/100]$ contains $\gg N^{1-2\eps}/\log N$ rational numbers of the form $c_p / p$, where $c_p \in \{1,2,\ldots,p-1\}$ and where $p$ is a prime in $[N, 2N]$.  To see this, note that for each prime $p \in [N,2N]$ there are $\sim p N^{-1/2-\eps}/20$ rationals $a/p \in [N^{-1/2-\eps}/20, N^{-1/2-\eps}/10]$.  So, there are on the order of $N^{-1/2-\eps} \sum_{y\leq p \leq 2y} p \gg N^{3/2-\eps}/\log N$ rationals $a/p$ in that interval; and then if we partition the interval $[N^{-1/2-\eps}/20, N^{-1/2-\eps}/10]$ into disjoint subintervals of width $N^{-1-2\eps}/100$ (except for the last sub-interval, which might be shorter), since there are about $N^{1/2+\eps}$ such intervals, the average one will contain $\gg N^{1-2\eps}/\log N$ of those rationals $a/p$, as claimed.  

Note that 
if $d_p / p \in I$ then no other rational with denominator $p$ is contained in $I$; so, given $p$ the $d_p$ is uniquely determined.  Let ${\mathcal P}$ denote the set of primes $p\in [N,2N]$ such that there is $d_p\in \{1,2,\ldots, p-1\}$ such that $d_p/p \in I$; moreover, for each such $p$, let $e_p\in \{1,2,\ldots, p-1\}$ such that $d_p e_p \equiv 1 \pmod d$, and define $R_p=\{je_p: 0 \leq j \leq \lfloor p^{1/2-\eps}\rfloor\}$. 

Next, we note that if
$n \in \{0,1,2,3,4\}$ and $p \in {\mathcal P}$, then $d_p/p \in I$ and thus
$$0\leq n d_p \leq 4\bigg(\frac{N^{-1/2-\eps}}{10}+\frac{N^{-1-2\eps}}{100}\bigg)p \leq \frac{4}{5}p^{1/2-\eps}+\frac{1}{5p^{2\eps}}.$$
It follows that $\{0,1,2,3,4\} \subseteq A$. However, if $20\leq n \leq 30$, then a similar computation shows that for each $p \in {\mathcal P}$, we have
$p^{1/2-\eps}<n d_p<p$.  It follows that $A \cap [20,30]=\emptyset$.

But now consider an integer $m$ 
such that $x^{-1} \leq m < x^{-1}+1$.  Write this $m$ as $m = x^{-1} + \delta$,
where $\delta \in [0,1)$. For each $p \in \mathcal{P}$, we can write $d_p/p = x + \gamma$ with $0 \leq \gamma \leq N^{-1-2\eps}/100$, and thus we have 
$$
0\leq -1+md_p / p =\ -1+(x^{-1}+\delta)(x + \gamma)=x^{-1} \gamma+ \delta x + \delta \gamma \ <\ N^{-1/2 -\eps}/2.
$$
Thus $m\in A$. So, the set $A$ is clearly not an arithmetic progression -- looking at the smaller values of $n$ (say, $0 \leq n \leq 4$), if $A$
were an arithmetic progression, then it 
would have to be an interval (step size of the progression is $1$); but then there is that gap $[20,30]$ containing no elements of $S$; and finally, the elements in the set $A$ resume when $m$ is a little bigger than $x^{-1}$.  

\subsection{A construction showing that $S$ can be an arithmetic progression stretching nearly to $N$} \label{extraremark2}\

As we have seen in Section~\ref{sec:2.1}, choosing $R_p=[0,\lfloor p^{1/2-\eps}\rfloor] \cap \Z_p$ for each prime $p\in [N,2N]$ results in a set $A$ that is an interval
$\{0,1,2, \ldots, M\}$, where $M \sim N^{1/2-\eps}$.  Could we perhaps choose the $R_p$ so that $A$ is an arithmetic progression ending at $N$ or near $N$?  If it turned out to be the case that all the sets $A$ we could construct were contained in a short interval like $[0, N^{1/2-\eps}]$, then that would suggest Theorem~\ref{maintheorem} could be strengthened.  This is not the case, however.

Let $D$ to be the smallest prime $> 2N^{1/2+\eps}$, and consider the union of intervals
$$
I\ :=\ \bigcup_{a=0}^{D-1} \left [{a\over D},\ 
{a\over D} + {1 \over D N}\right ].
$$
For every prime $p \in [N, 2N]$ we claim that for some choice of
$d_p$ we will have that $d_p / p$ that is congruent to an element
of $I$ mod $1$.  To see this, we begin by noting that since
$\gcd(p,D)=1$, there exist $a_p$ and $d_p$, where
$0 \leq a_p \leq D-1$ and $0 \leq d_p \leq p-1$, such that
$$
d_p D- a_p p \ =\ 1.
$$
Thus,
$$
0<{d_p \over p}- {a_p \over D} \ =\ {1 \over Dp}
\ \leq\ {1 \over DN}.
$$
It follows from this that $d_p/p$ is congruent to an element of $I$
mod $1$; and given $p$ we will choose $d_p$ so that this holds. Moreover, for prime $p \in [N, 2N]$, let $e_p\in \{1,2,\ldots, p-1\}$ such that $d_p e_p \equiv 1 \pmod d$, and define $R_p=\{je_p: 0 \leq j \leq \lfloor p^{1/2-\eps}\rfloor\}$.

Now, if $1 \leq n \leq N$ is a multiple of $D$, then
we can write it as $n = Dm$, where $0 \leq m \leq N/D$. For a given prime $p \in [N,2N]$
we note $d_p/p = a_p/D + \delta$, where $0<\delta \leq 1/DN$; and we have
$$
0<{n d_p \over p}-ma_p=\frac{nd_p}{p}-\frac{na_p}{D}=n\delta= Dm \delta\leq \frac{m}{N}\leq\ {1\over D}\ <\ {1 \over 2N^{1/2+\eps}}<\frac{1}{p^{1/2+\eps}}.
$$
It follows that $n\in A$.

On the other hand, if $n = Dm+r$, $1 \leq r \leq D-1$, then upon
choosing $p \in [N, 2N]$ to be any prime where $r a_p \not \equiv 
\pm 1, \pm 2, \pm 3 \pmod{D}$, we would have that
\begin{eqnarray*}
\left \|{n d_p \over p} \right \|\ =\ \left \| (Dm+r)\delta + {r a_p \over D} \right \|\ &\geq&\ \left \| {r a_p \over D} \right \|
- {1 \over D} - {1 \over N}\\
&\geq&\ {4 \over D} - {1 \over D} - 
{1 \over N}\ =\ {3 \over D} - {1 \over N}
>\ {1 \over N^{1/2+\eps}}.
\end{eqnarray*}
Thus, $n \not \in A$ in this case.

One loose end to clear up in this proof is that there 
even exist primes where $r a_p \not \equiv \pm 1, \pm 2, \pm 3 \pmod{D}$.  This is not a problem, since we can easily bound from above the number
of primes $p$ where $r a_p \in \{\pm 1, \pm 2, \pm 3\} \pmod{D}$. Recall that $a_p p \equiv -1 \pmod D$, it follows that if $r a_p \in \{\pm 1, \pm 2, \pm 3\} \pmod{D}$, then $r \in \{\pm p, \pm 2p, \pm 3p\} \pmod{D}$, that is, such $p$ necessarily lies in the union of $6$ fixed progressions with step size $D$. Thus, there are at most $6N/D$ primes 
$p \in [N,2N]$ where $r a_p \equiv \pm 1, \pm 2, \pm 3 \pmod{D}$,
meaning there are at least $\pi([N,2N]) - 6N/D \sim \pi([N,2N])$ 
primes we could choose from.

\subsection{Theorem~\ref{maintheorem} does not extend to polynomial progressions}\

Another possible direction to generalize Theorem~\ref{maintheorem} is to consider sets $R_p$ defined by polynomials. Specifically, one might take $R_p = f_p(I_p)$ for some $f_p\in\Z_p[x]$ and $I_p=\{0,1,2,\ldots,\lfloor p^c\rfloor\}$. The aim would then be to conclude that the set $A$ must be of the form $F(\{0,1,2,\ldots,K\})$, where $F\in\Z[x]$ has degree $\mathrm{deg}(F)\leq \max\{\mathrm{deg}(f_p):p\in [y,2y]\}$ and $K=O(y^c)$. Note that Theorem~\ref{maintheorem} corresponds to the case when all the $f_p$'s are linear. 

Here we provide a construction demonstrating that when $f_p$'s have degree $d = 2$, the resulting set $A$ may not even be a dense subset of $F(\{0,1,\ldots,K\})$ for any quadratic polynomial $F\in\Z[x]$. Let $y=N/2$. Let $p_0$ be the largest prime in $[y,2y]$ and $R_{p_0} = \{2m^2+1 : 0\leq m\leq p_0^{1/2-\eps}\}\subseteq\Z_{p_0}$; for all other primes $p\in [y,2y]$, let $R_p = \{n^2 : 0\leq n\leq p^{1/2-\eps}\}\subseteq\Z_p$. Then $A\subseteq\{n^2 : 0\leq n\leq q^{1/2-\eps}\},$
where $q$ is the smallest prime in $[y,2y]$. For any $n^2\in A$, there exists some $0\leq m\leq p_0^{1/2-\eps}$ such that $n^2\equiv 2m^2+1\pmod {p_0}$. Since $0\leq n^2,2m^2+1<p_0$, this congruence must be an equality in $\Z$: $n^2=2m^2+1$. It is well-known that the solutions to this Pell equation grow exponentially and $|A|\asymp\log N$, thus $A$ cannot be a dense subset of $F(\{0,1,\ldots,K\})$ for any quadratic polynomial $F\in \Z[x]$ and $K\leq y^{1/2-\eps}$.

\section{Proof of Theorems~\ref{maintheorem} and~\ref{thm:interval}} \label{sec:theoremproofsection}

In this section, we first prove Theorem~\ref{maintheorem} and Corollary~\ref{cor:allprimes}, and then discuss how to use a similar method to prove Theorem~\ref{thm:interval}.

We first reduce Theorem~\ref{maintheorem} to the following symmetric version.
\begin{prop}\label{mainprop} 
There is an absolute constant $C>0$ such that the following holds. 
Let $\eps \in (0, 1/2)$.  Suppose $N$ is a positive integer with $N> N_0(\eps)$, and $y$ satisfies
\begin{equation*}
(4C\log N)^{1/2\eps}\ \leq\ y\ \leq\ N.
\end{equation*}
Let ${\mathcal P}$ be a subset of primes in $[y,2y]$ satisfying 
$$
|{\mathcal P}|\ \geq {Cy^{1-2\eps} \log N \over \log y}.
$$
For each prime $p \in \mathcal{P}$, choose $d_p$ be some
integer in $\{1,2,\ldots,p-1\}$.  Consider
the set $S$ of all integers $0 \leq n \leq N$ with the property that for all primes $p \in \mathcal{P}$, 
$n d_p\ \in\ [\ell_p, r_p] \pmod{p},$ where
$-p^{1/2-\eps} \leq \ell_p\ \leq\ 0
\ \leq\ r_p\ \leq\ p^{1/2-\eps}.$
Then, $S$ is an arithmetic progression of length at most $(2y)^{1/2-\eps} + 1$. 
\end{prop}

Next, we show that Proposition~\ref{mainprop} implies Theorem~\ref {maintheorem}.

\begin{proof}[Proof of Theorem \ref{maintheorem} assuming Proposition~\ref{mainprop}]
Assume that $A$ precisely consists of elements $n \in [0, N]$ with the property that
$$
n \in  \{\delta_p+j e_p: \ell_p\leq j \leq r_p\} \pmod p
$$
for all primes $p \in \mathcal{P}$, where $\delta_p, \ell_p, r_p$ are integers with $0\leq r_p-\ell_p\leq p^{1/2-\eps}$. 

If $A$ is empty, then we are done. Assume otherwise that $A$ is non-empty and pick the smallest element $a$ in $A$. Then for each prime $p \in \mathcal{P}$, we can find an integer $b_p \in [\ell_p, r_p]$ such that $a \equiv \delta_p +b_p e_p \pmod p$. For each prime $p \in \mathcal{P}$, let $d_p\in \{1,2, \ldots, p-1\}$ such that $e_p d_p \equiv 1 \pmod p$. It follows that
$A-a$ precisely consists of elements $n \in [0, N-a]$ with the property that
$$
nd_p  \in  [\ell_p-b_p,r_p-b_p] \pmod p, \quad \forall p \in \mathcal{P}.
$$
Note that we have $-p^{1/2-\eps}\leq \ell_p-b_p \leq 0\leq r_p-b_p\leq p^{1/2-\eps}$ for each prime $p \in \mathcal{P}$. Thus, the theorem follows from Proposition~\ref{mainprop}. 
\end{proof}

Assume that we are in the setting of Proposition~\ref{mainprop}. We need two lemmas.  The first lemma is as follows.

\begin{lemma}\label{lem2}
Let $M$ be a positive integer. Suppose we have an arithmetic progression $P$ of length at most $M^{1/2}/10$, and a prime $p\in [M, 2M]$. Let $Q$ be the set of integers that mod $p$ belong to a progression $\{a,a+d, a+2d, \ldots, a+kd\}$, where $k<p^{1/2}/10$. Then $P \cap Q$ is either empty or is an arithmetic progression.   
\end{lemma}

\begin{proof}
Suppose
$$P = \{b, b+e, b+2e, \ldots, b+ne\}, \quad n < M^{1/2}/10.$$
If $p \mid e$, then the lemma follows trivially. Next, assume that $ p\nmid e$. Then, our intersection $P \cap Q$ will include exactly those $b+je$ with $0\leq j \leq n$ such that there is $0 \leq m\leq k$ such that
\begin{equation}\label{eq:equal}
b+je \equiv  a+md \pmod p. 
\end{equation}

Now, by the pigeonhole principle, there exists an integer $1\leq x\leq p-1$ such that
$$
\left \|\frac{xe}{p}\right \|<\frac{1}{\sqrt{p}}, \quad \left \| \frac{xd}{p}\right \|<\frac{1}{\sqrt{p}}.
$$
Let $e' \equiv xe \pmod p, d' \equiv xd \pmod p$, where $e'$ and $d'$ are the smallest residues in absolute value. By the above assumption, we have $|d'|,|e'|<\sqrt{p}$. 

Then, multiplying both sides of equation (\ref{eq:equal}) by $x$ and rearranging, we get the following equivalent equation
$$je' - md' \equiv x(a-b) \pmod p.$$
Note that $$|je'-md'|\leq n|e'|+k|d'|\leq p/5<p/2.$$ 
Thus, if we let $z$ denote the smallest residue in absolute value that is $\equiv x(a-b) \pmod p$, then we have
$$je' - md'= z.$$
It is well-known that solutions $(j,m)$ to this equation form an arithmetic progression. This proves the claim.    
\end{proof}

\begin{corollary}\label{cor2}
Let $y$ and $S$ be as in the statement of
Proposition~\ref{mainprop}.  Suppose $P\subseteq \{0,1,\ldots, N\}$ is an arithmetic progression with length at most $y^{1/2}/10$. Then $P \cap S$ is an arithmetic progression.
\end{corollary}
\begin{proof}
List the primes in $\mathcal{P}$ by $p_1, p_2, \ldots, p_m$. For each $1\leq j \leq m$, let 
$$
Q_{j}=\{0\leq n \leq N: n d_{p_j} \in [\ell_{p_j}, r_{p_j}] \pmod {p_j}\}.
$$
Let $P_0=P$. For each $1\leq j \leq m$, let $P_j=P_{j-1} \cap Q_j$. Note that $S=\bigcap_{j=1}^m Q_j$ and thus $P_m=P \cap S$. Using Lemma~\ref{lem2}, it is easy to prove by induction that $P_j$ is an arithmetic progression with length at most $y^{1/2}/10$ for each $0\leq j \leq m$. In particular, $P \cap S=P_m$ is an arithmetic progression.
\end{proof}

We also need the following lemma:
\begin{lemma}\label{lem:gcd}
Let $y$ and $S$ be as in the statement of Proposition~\ref{mainprop}.
Let 
$$
a_1,a_2, \ldots, a_m, b_1, b_2, \ldots, b_m \in S
$$
and let 
$
\delta_1, \delta_2, \ldots, \delta_m, \eta_1, \eta_2, \ldots, \eta_m
$
be integers with 
$$
\sum_{i=1}^m |\delta_i|+\sum_{i=1}^m |\eta_i|\ \leq 10.
$$
Write 
$
a=\sum_{i=1}^m \delta_i a_i\ {\rm and\ } b=\sum_{i=1}^m \eta_i b_i.
$
Then 
$$
\frac{a}{\gcd(a,b)}\leq \sum_{i=1}^m |\delta_i| \cdot (2y)^{1/2-\eps}.
$$
\end{lemma}
\begin{proof}
Let $p\in \mathcal{P}$ be a fixed prime. Let $e_p\in \{1,2, \ldots, p-1\}$ such that $e_p d_p \equiv 1 \pmod p$. Then for each $1\leq i \leq m$, there are integers $s_i, t_i \in [-p^{1/2-\eps},p^{1/2-\eps}]\subseteq [-(2y)^{1/2-\eps}, (2y)^{1/2-\eps}]$ such that $a_i \equiv s_i e_p \pmod p$ and $b_i \equiv t_i e_p \pmod p$. Let $s=\sum_{i=1}^m \delta_i s_i$ and $t=\sum_{i=1}^m \eta_i t_i$.  We note that $|s|+|t|<40y^{1/2-\eps}$ 
and that
$$
e_p s\ \equiv\ \sum_{i=1}^m \delta_i a_i
\ \equiv\ a \pmod{p},
$$
and 
$$
e_p t\ \equiv\ \sum_{i=1}^m \eta_i b_i
\ \equiv\ b \pmod{p}.
$$
It follows that $p \mid (ta - sb)$. 

The above argument shows that for each prime $p\in \mathcal{P}$, there are integers $s_p, t_p$ such that $|s_p|+|t_p|<40y^{1/2-\eps}$ and $p\mid (t_p a-s_p b)$. Note that there is an absolute constant $C'$ such that the number of possible integer pairs $(s_p,t_p)$ with $|s_p|+|t_p|<40y^{1/2-\eps}$ is at most $C'y^{1-2\eps}$. Let $C=2C'$. Then, by the pigeonhole principle, we can find some integers $s$ and $t$ with $|s|+|t|<40y^{1/2-\eps}$ such that $t a-sb$ is divisible by at least $|\mathcal{P}|/C'y^{1-2\eps}$ different primes in $\mathcal{P}$. Thus, if $t a \neq sb$, then
$$
|t a -sb |\geq y^{\frac{|\mathcal{P}|}{C'y^{1-2\eps}}}\geq y^{\frac{2\log N}{\log y}}=N^2.
$$
However, we have $|t a-sb|<N^2$, which forces $t a=sb$. It follows that 
\[
\frac{a}{\gcd(a,b)}\leq |s|\leq \sum_{i=1}^m |\delta_i| \cdot (2y)^{1/2-\eps}.\qedhere
\]
\end{proof}

Now we are ready to complete the proof of Theorem~\ref{maintheorem} by proving Proposition~\ref{mainprop}.

\begin{proof}[Proof of Proposition~\ref{mainprop}]
Observe that $0\in S$. If $S=\{0\}$, then we are done. Assume that $S$ has a nonzero element, say $a_0$. Then for each $b\in S$, by Lemma \ref{lem:gcd} (using the case $m=1$ and setting $\delta_m = \delta_1 = \eta_m = \eta_1 = 1$ and $a=a_0$), $b \in d \{0,1,\ldots, (2y)^{1/2-\eps}\}$, where $d=\gcd(a_0,b)$. In particular,
$$
S\subseteq \bigcup_{d \mid a_0} d \{0,1,\ldots, (2y)^{1/2-\eps}\}:=\bigcup_{i=1}^k P_i.
$$
For each $1\leq i \leq k$, since $P_i$ is an arithmetic progression starting from 0 with length at most $(2y)^{1/2-\eps}$, Corollary \ref{cor2} implies that $Q_i:=P_i \cap S$ is an arithmetic progression starting from $0$ with length at most $(2y)^{1/2-\eps}$. Note that we have $S=\bigcup_{i=1}^k Q_i$. 

Next, we show that $S$ is an arithmetic progression. If $k=1$, then we are done. Next, assume that $k\geq 2$. We claim that $Q_1 \cup Q_2=Q$ for some arithmetic progression $Q$ starting from $0$ with length at most $4y^{1/2-\eps}$. If $Q_1=\{0\}$ or $Q_2=\{0\}$, this is obvious. Next assume that $|Q_1|, |Q_2|\geq 2$. Say $Q_1=\{0, s_1, 2s_1, \ldots, L_1s_1\}$ and $Q_2=\{0, s_2, 2s_2, \ldots, L_2s_2\}$. It follows that $Q_1 \cup Q_2\subseteq Q':=\{0, s, 2s, \ldots, \ell s\}$, where $s=\gcd(s_1,s_2)$ and $\ell s \leq \max \{L_1 s_1, L_2 s_2\}$. Pick 
$$
a_1=L_1 s_1, a_2=s_2, b_1=(L_1-1)s_1, b_2=s_2, \delta_1=\delta_2=\eta_1=\eta_2=1
$$
in Lemma \ref{lem:gcd}, we obtain that
$$
4y^{1/2-\eps}\geq \frac{L_1 s_1+s_2}{\gcd(L_1 s_1+s_2, (L_1-1)s_1+s_2)}\geq \frac{L_1 s_1}{s},
$$
that is, $L_1s_1/s\leq 4y^{1/2-\eps}$. Similarly, $L_2 s_2/s\leq 4y^{1/2-\eps}$. This shows that $\ell \leq 4y^{1/2-\eps}$. Since $Q'$ is an arithmetic progression starting from $0$, with length at most $4y^{1/2-\eps}$, Corollary \ref{cor2} implies that $Q:=Q' \cap S=Q_1\cup Q_2$ satisfies the claim. Thus, we have shown that when $m=2$, $S$ is an arithmetic progression. If $k\geq 3$, we can apply the above argument inductively to conclude that $S$ is an arithmetic progression. 

Finally, since $S$ is an arithmetic progression starting from $0$, we can write $S=\{jd: 0\leq j \leq \ell\}$. If $S\neq \{0\}$, that is, $\ell \geq 1$, then by Lemma~\ref{lem:gcd}, we have
$$
\ell=\frac{\ell d}{\gcd(\ell d, (\ell-1)d)}\leq (2y)^{1/2-\eps}.
$$
Thus, $S$ has length at most $(2y)^{1/2-\eps}+1$, as required.
\end{proof}

Next, we use results on the distribution of primes in short intervals to prove Corollary~\ref{cor:allprimes}.

\begin{proof}[Proof of Corollary~\ref{cor:allprimes}]
Let $C$ be the absolute constant from Theorem~\ref{maintheorem}. Note that we have $y^{2\eps}\geq 4C\log N$.

By Theorem~\ref{maintheorem}, $A$ is an arithmetic progression of length at most $(2y)^{1/2-\eps}$. Assume that $|A|\geq 2$, for otherwise we are done. Let $d$ be the step size of the arithmetic progression $A$ and let $z$ be the smallest prime in $[y,2y]$ such that $z \nmid d$. Then $A_z$ is an arithmetic progression in $\Z_p$ of length $|A|$ and it follows that $|A|\leq z^{1/2-\eps}+1$. 

On the other hand, since $$N\geq d\geq \prod_{y\leq p<z} p \geq y^{\pi(z-1)-\pi(y)},$$ it follows that $\pi(z-1)-\pi(y)\leq \log N/\log y$. To establish an upper bound on $z$, we use results of Huxley \cite{Huxley} (see also the recent breakthrough of Guth and Maynard \cite{GM}) on counting the number of primes in short intervals. We consider two cases based on the value of $\eps$. 
\begin{enumerate}
    \item $\eps\geq 0.3$. In this case, by \cite{Huxley} or \cite[Corollary 1.3] {GM}, for sufficiently large $y$, $$\pi(y+y^{2\eps}/C)-\pi(y)\geq y^{2\eps}/2C\log y\geq 2\log N/\log y.$$ Since $\eps<1/2$, it follows that $z=(1+o(1))y$ for sufficiently large $N$.
    \item $\eps<0.3$. In this case, by \cite[Corollary 1.4] {GM}, when $y$ is sufficiently large, we have $$\pi(y+y/\exp((\log y)^{1/5})+y^{0.8})-\pi(y)\geq \frac{y^{0.8}}{2\log y}> \frac{2\log N}{\log y}.$$ This again implies that $z=(1+o(1))y$ for sufficiently large $N$.
\end{enumerate}

To conclude, we can find an absolute constant $C_0$, such that if $N\geq N_0(\eps)+C_0$, then $z=(1+o(1))y$ and thus $|A|\leq z^{1/2-\eps}+1=(1+o(1))y^{1/2-\eps}$.
\end{proof}

We end the section with a short proof of Theorem~\ref{thm:interval}.

\begin{proof}[Proof of Theorem~\ref{thm:interval}]

Assume that $A$ precisely consists of elements $n \in [0, N]$ with the property that
$$
n \in  \{\delta_p+j : 0\leq j \leq r_p\} \pmod p
$$
for all primes $p \in \mathcal{P}$, where $\delta_p, r_p$ are integers with $0\leq r_p\leq p^{1-\eps}$. 

If $A$ is empty, then we are done. Assume otherwise that $A$ is non-empty and pick the smallest element $a_1$ in $A$. Then for each prime $p \in \mathcal{P}$, we can find an integer $b_p \in [0, r_p]$ such that $a_1 \equiv \delta_p +b_p \pmod p$. It follows that
$A-a_1$ precisely consists of elements $n \in [0, N-a_1]$ with the property that
$$
n  \in  [-b_p,r_p-b_p] \pmod p, \quad \forall p \in [y,2y].
$$
Note that we have $-p^{1-\eps}\leq -b_p \leq 0\leq r_p-b_p\leq p^{1-\eps}$. Let $a_2$ be the largest element of $A$. Then for each prime $p\in \mathcal{P}$, there exists some integer $\alpha_p$ such that $|\alpha_p|\leq (2y)^{1-\eps}$ and $a_2-a_1\equiv \alpha_p\pmod p$. By the pigeonhole principle, we can find some integer $\alpha$ with $|\alpha|\leq (2y)^{1-\eps}$ such that $a_2-a_1\equiv\alpha\pmod p$ for at least $|\mathcal{P}|/(2\cdot(2y)^{1-\eps}+1)$ different primes in $\mathcal{P}$. Thus, if $a_2-a_1\neq \alpha,$ then 
\[|a_2-a_1-\alpha|\geq y^{\frac{|\mathcal{P}|}{2(2y)^{1-\eps}+1}}\geq y^{\frac{3\log N}{2\log y}} = N^{3/2}.\]
However, we know $|a_2-a_1-\alpha|<N^{3/2}$, which forces $a_2-a_1 = \alpha$. Now for any $p\in \mathcal{P}$, $a_2-a_1\leq (2y)^{1-\eps}<p/4$, which guarantees that $[a_1,a_2]$ is the shortest interval containing $\{a_1,a_2\}$ and hence $I_p\supseteq [a_1,a_2]\cap\Z_p$. It follows that $A = [a_1,a_2]\cap\Z$. In particular, $|A| = a_2-a_1\leq |R_{p_0}|\leq p_0^{1-\eps}+1$, as required. 
\end{proof}

\section{Longer progressions: Proof of Theorems~\ref{thm:AP1/2} and~\ref{thm:largeA}}\label{sec:1/2}

In this section, we prove theorems~\ref{thm:AP1/2} and~\ref{thm:largeA}. While the proof of these two theorems shares some similarity with that of Theorem~\ref{maintheorem}, we need several extra ingredients and observations. 
\subsection{Proof of Theorem~\ref{thm:AP1/2}}
By translating $A$ and $A\cap Q$ we may assume that $$Q = \{0,a_1,2a_1,\ldots, (M-1)a_1\}$$ for some integer $M\geq y^c$, and $0\in A\subseteq [-N,N]$.     Let $a_2$ be an arbitrary nonzero element of $A$. 
Let $$\delta = 8y^{\frac{2\theta-1-c}{2}+\eps}.$$ From the assumption that $|A\cap Q|\geq \delta|Q|>y^{\frac{2\theta-1+c}{2}}\geq 1$, there exist $s,s'\in A\cap Q$ such that $s-s' = ta_1$ for some integer $1\leq t\leq \delta^{-1}$ by pigeonhole. 

Let $B = (A\cap Q)/a_1$, which is a subset of $\{0,1,\ldots,M\}$. Since $|B|\geq \delta M$, it follows from Lemma~\ref{lem:lev1} that for some positive integer $h \leq \lceil4\delta^{-1}\rceil$, the $h$-fold sumset $hB$ contains an arithmetic progression $$Q' = \{mb_1,(m+1)b_1,\ldots,(m+M)b_1\},$$ where $0<b_1\leq 4\delta^{-1}$. In particular, $p\nmid b_1$ for all primes $p\in [y,2y]$.   
    
    Fix a prime $p\in [y,2y]$. Suppose $R_p = \{\alpha_p, \alpha_p+u_p,\ldots,\alpha_p+\ell_pu_p\}$ for some $\ell_p\leq p^\theta$. Since $0\in A$, we have $\alpha_p \equiv m_pu_p\pmod p$ for some $-\ell_p\leq m_p\leq 0$. In particular, $R_p \subseteq \{ju_p: -\ell_p\leq j \leq \ell_p\}$.    
    It follows that there exist two integers $k_1^{(p)},k_2^{(p)}$ with $|k_i^{(p)}|\leq \ell_p$ so that 
    \begin{equation}\label{eq:cong1}
    ta_1\equiv k_1^{(p)}u_p\pmod p\quad\text{and}\quad a_2\equiv k_2^{(p)}u_p\pmod p.
    \end{equation}
    Next, we use the progression $Q'$ to deduce a better bound on $|k_1^{(p)}|$. For any $0\leq j\leq M$, since $(m+j)b_1\in hB$, we have
    $(m+j)b_1a_1 \in h(A \cap Q)$ and thus there is an integer $n_j$ with $|n_j|\leq h\ell_p$ such that 
    \[(m+j)b_1a_1\equiv n_ju_p\pmod p\quad \Rightarrow \quad (m+j)b_1k_1^{(p)}\equiv n_jt\pmod p.\]
    From size consideration, we get $|n_jt|\leq h\ell_p\delta^{-1}\leq 8\ell_p\delta^{-2}<p/4$. Thus, 
    \[
    j(n_1-n_0)t\equiv jb_1k_1^{(p)}\equiv (n_j-n_0)t\pmod p,
    \] 
    which implies $|(n_1-n_0)t|\leq \frac{16\ell_p}{\delta^2M}$. We also know that $|b_1k_1^{(p)}|\leq 4\delta^{-1}\ell_p<p/2,$ hence $$b_1k_1^{(p)} = (n_1-n_0)t \quad \text{and} \quad |k_1^{(p)}|\leq \frac{16\ell_p}{\delta^2|b_1|M} \leq \frac{16(2y)^\theta}{\delta^2M}.$$

    The above argument shows that for each prime $p\in [y,2y]$, there are integers $k_1^{(p)},k_2^{(p)}$ such that 
    \begin{equation}\label{eq:k1p}
    |k_1^{(p)}|\leq \frac{16(2y)^\theta}{\delta^2M} \quad \text{and} \quad |k_2^{(p)}|\leq (2y)^\theta,
    \end{equation} with $p\mid k_2^{(p)}ta_1-k_1^{(p)}a_2$ by relations~\eqref{eq:cong1}.      Note that the number of such integer pairs $(k_1^{(p)},k_2^{(p)})$ is at most $\frac{64(2y)^{2\theta}}{\delta^2M}$. It follows from  the prime number theorem and the pigeonhole principle that we can find some integers $k_1$ and $k_2$ satisfying the estimates~\eqref{eq:k1p}, such that $k_2ta_1-k_1a_2$ is divisible by all the primes in a subset of primes $\mathcal{P}\subseteq [y,2y]$ with $$|\mathcal{P}|\geq \frac{y}{2\log y}\cdot \frac{\delta^2M} {64(2y)^{2\theta}}>   
    \frac{\delta^2My^{1-2\theta}}{512\log y}\geq \frac{y^{2\eps}}{8\log y}.$$ If $k_2ta_1\neq k_1a_2$, then the assumption $y\geq (16\log N)^{1/2\eps}$ implies that
    \[|k_2ta_1-k_1a_2|\geq y^{|\mathcal{P}|}\geq \exp(y^{2\eps}/8)\geq \exp(2\log N)= N^2.\]
    However, we have $|k_2ta_1-k_1a_2|<N^2$, which forces $k_2ta_1=k_1a_2$. If $k_2 = 0$, then from relations~\eqref{eq:cong1} we have $p\mid a_2$ for all $p\in \mathcal{P}$, which implies $a_2 = 0$, a contradiction. Let $d = \gcd(a_1,a_2)$, $a_i' = a_i/d$ for $i=1,2$. Then we must have 
    \begin{equation}\label{eq:lambda}
    tk_2 = \lambda a_2' \quad \text{and} \quad  k_1 = \lambda a_1'
    \end{equation}
    for some nonzero integer $\lambda$.  It follows from our choice of $d$ that $(A\cap Q)\cup \{a_2\}$ is contained in a longer arithmetic progression $\{-Ld,\ldots,-d,0,d,\ldots, Ld\}$, where $L\leq \max(M|a_1'|,|a_2'|)$. 

\medskip

    Next, we fix a prime $p\in [y,2y]$ and bring these back to the congruence relations~\eqref{eq:cong1}. One can deduce that 
    \begin{equation}\label{eq:reducek1p}
    ta_1'd\equiv k_1^{(p)}u_p\pmod p\quad \text{and}\quad a_2'd\equiv k_2^{(p)}u_p\pmod p,
    \end{equation}
    which implies
    \[
    k_2^{(p)}ta_1'd\equiv k_1^{(p)}a_2'd\pmod p.
    \]
    Next, we consider two cases.
    
    \textbf{Case 1: $p\nmid d$.} Then equation~\eqref{eq:lambda} implies that
    \[
    \lambda tk_2^{(p)}a_1'\equiv \lambda k_1^{(p)} a_2'\pmod p\quad \Rightarrow\quad tk_2^{(p)}k_1\equiv tk_1^{(p)}k_2\pmod p.
    \]
    From estimates~\eqref{eq:k1p}, it follows that $k_2^{(p)}k_1 = k_1^{(p)}k_2$. Therefore, equation~\eqref{eq:lambda} implies that $ta_1'k_2^{(p)}=a_2'k_1^{(p)}$. Since $a_1'$ and $a_2'$ are coprime, there is some nonzero integer $v_p$ such that    
    \begin{equation}\label{eq:k1k2}
    tk_2^{(p)} = v_pa_2' \quad \text{and} \quad k_1^{(p)} = v_pa_1'.
    \end{equation}
    For simplicity, we may assume $v_p>0$; the other case is similar. Combined with relations~\eqref{eq:reducek1p}, we get 
    \[
    ta_1'd\equiv v_pa_1'u_p\pmod p\quad \text{and}\quad ta_2'd\equiv v_pa_2'u_p\pmod p.
    \]
    Since $\gcd(a_1',a_2') = 1$, we must have $td\equiv v_p u_p\pmod p$.
    
    Suppose now $jd\in (A\cap Q)\cup\{a_2\}$. Then there is an integer $0\leq i\leq \ell_p$ so that $jd\equiv (m_p+i)u_p\pmod p$. Therefore,
    \begin{equation}\label{eq:mp+i}
    (m_p+i)\frac{t}{\gcd(t,v_p)}u_p\equiv j\frac{t}{\gcd(t,v_p)}d\equiv j\frac{v_p}{\gcd(t,v_p)}u_p\pmod p.
    \end{equation}
 Since $|j|\leq L\leq \max(M|a_1'|, |a_2'|)$, by estimates~\eqref{eq:k1p} and~\eqref{eq:k1k2}, we have either 
    $$
    |jv_p|\leq L|v_p|\leq M|a_1'||v_p|=M|k_1^{(p)}|\leq 16(2y)^{\theta} \delta^{-2}
    $$
    or
    $$
    |jv_p|\leq L|v_p|\leq |a_2'||v_p|=t|k_2^{(p)}|\leq (2y)^{\theta} \delta^{-1}.
    $$
    In both cases, we have $$|jv_p|\leq 16(2y)^{\theta} \delta^{-2}\leq \frac{1}{2} y^{1+c-\theta-2\eps}<\frac{y}{2}.$$
    Observe that for the left-hand side of equation~\eqref{eq:mp+i}, we have $|(m_p+i)t|\leq \ell_p\delta^{-1}<p/2$. It follows that $j\frac{v_p}{\gcd(t,v_p)} = (m_p+i)\frac{t}{\gcd(t,v_p)}$ and hence $j\in I_p\cdot\frac{t}{\gcd(t,v_p)}$, where
    \[
    I_p = \bigg[ \frac{m_p}{v_p/\gcd(t,v_p)}, \frac{m_p+\ell_p}{v_p/\gcd(t,v_p)}\bigg]\bigcap\Z.
    \]
    On the other hand, for any $j\in I_p\cdot\frac{t}{\gcd(t,v_p)}$, 
    \[
    jd\equiv \frac{wtd}{\gcd(t,v_p)}\equiv w\frac{v_p}{\gcd(t,v_p)}u_p\pmod p
    \]
    for some $w\in I_p$, which immediately implies $jd\pmod p\in R_p$. 

    \textbf{Case 2: $p\mid d$.} In this case, we have $(A\cap Q)\cup\{a_2\}\subseteq \{-Ld,\ldots,-d,0,d,\ldots, Ld\}$ and $\{-Ld,\ldots, Ld\}\pmod p=\{0\}\subseteq R_p$ trivially.

\medskip
    
    To conclude, we have obtained a set 
    \[
    Q_0 = \{-Ld,\ldots,-d,0,d,\ldots, Ld\} \cap \bigcap_{p\in [y,2y],\ p\nmid d}\bigg(I_p\cdot\frac{t}{\gcd(t,v_p)}\bigg) 
    \]
    that satisfies $(A\cap Q)\cup\{a_2\}\subseteq Q_0$ and $Q_0\pmod p\subseteq R_p$ for each $p\in [y,2y]$. In particular, this means $Q_0\subseteq A$. Note that $Q_0$ is given by the intersection of several truncations of homogeneous arithmetic progressions, which forces itself to be an arithmetic progression. If $Q_0\neq A$, then we can replace $A\cap Q$ by $Q_0$ and iterate. This iteration must terminate, from which we conclude that $A$ must be an arithmetic progression.

\subsection{Proof of Theorem~\ref{thm:largeA}}\

(1).
    Let 
    \[
    C = \frac{2^{1-\theta}-1}{1-\theta},\quad\delta = 64\sqrt{2}y^{\frac{\theta-1}{2}+\eps} \geq 64C^{-1/2}y^{\frac{\theta-1}{2}+\eps}.
    \]
    Since $N$ is large compared to $\theta$ and $\eps$, by the given lower bound on $y$, we can assume that $y$ is large enough. 
    For any positive integer $h$ and any prime $p\in [y,2y]$, we have $(hA)_p\subseteq hR_p$. Since $R_p$ is an arithmetic progression, it follows that $|hR_p|\leq h|R_p|\leq hp^\theta$. Now by Lemma~\ref{GS} and the prime number theorem, 
    \begin{align}
        |hA|
        &\leq \frac{\sum_{p\in [y,2y]}\log p-\log (hN)}{\sum_{p\in [y,2y]}\frac{\log p}{|(hA)_p|}-\log (hN)}\leq \frac{2y}{\sum_{p\in [y,2y]}\frac{\log p}{hp^\theta}-\log (hN)} \notag\\
        &\leq \frac{2y}{\frac{2^{1-\theta}-1}{2h(1-\theta)}y^{1-\theta}-\log (hN)}\leq \frac{8h}{C}y^\theta \label{eq:Gsieve_cond2}, 
    \end{align}
    provided that \begin{equation*}
    \frac{C}{h}y^{1-\theta} > 4\log (hN).
    \end{equation*}
    Set $h = \lceil 32/(C\delta) \rceil$. Since $y\geq (\log N)^{1/\eps}$, we have
    $$
    \frac{C}{h}y^{1-\theta}\gg \delta y^{1-\theta}\gg y^{\frac{1-\theta}{2}+\eps}\gg y^{\frac{1-\theta}{2}} \log N.
    $$
    Thus, when $N$ is sufficiently large compared to $\theta$ and $\eps$,   
    inequality \eqref{eq:Gsieve_cond2} holds. It follows that
    \[
    |hA|\leq \frac{8h}{C}y^\theta\leq \frac{h(h+1)}{3}\delta y^\theta\leq \frac{h(h+1)}{2}(|A|-2),
    \]
    and Lemma~\ref{lem:lev2} implies that there exists an arithmetic progression $Q$ such that $A\subseteq Q$ and $h|A|\geq |Q|$. In particular, since $\theta-\frac{1-\theta}{2}>2\theta-1$, when $N$ is sufficiently large compared to $\theta$ and $\eps$, we have  
    $|Q|\geq |A|\geq y^{2\theta-1}$.
    Let $c = \min\left(\frac{\log |Q|}{\log y},\theta\right)$, $\eta = |A|/|Q|\geq 1/h$. 

    If $|Q|\leq y^\theta$, then from our choice of $c$, $\delta$ and $h$, we have
    \[
    8y^{\frac{2\theta-1-c}{2}+\eps} = \frac{8y^{\frac{2\theta-1}{2}+\eps}}{|Q|^{1/2}} = 8\eta^{1/2}\frac{y^{\frac{2\theta-1}{2}+\eps}}{|A|^{1/2}} \leq 8\eta^{1/2}\frac{y^{\frac{\theta-1}{2}+\eps}}{\delta^{1/2}} \leq \frac{1}{8}(C\eta\delta)^{1/2} \leq \eta.
    \]
    It then follows from Theorem~\ref{thm:AP1/2} that $A$ must be an arithmetic progression.

    If $|Q| > y^\theta$, then $c = \theta$ and 
    \[8y^{\frac{2\theta-1-c}{2}+\frac{\eps}{2}} = 8y^{\frac{\theta-1}{2}+\frac{\eps}{2}} \leq \frac{1}{h}\leq \eta \]
    when $N$ is sufficiently large depending on $\theta$ and $\eps$. Again, we conclude from Theorem~\ref{thm:AP1/2} that $A$ must be an arithmetic progression. 

(2).
    Let 
    \[
    C = \frac{2^{1-\theta}-1}{1-\theta},\quad\delta = |A|/y^\theta\geq  80y^{\theta-1+\eps},\quad h = \lceil 32/(C\delta) \rceil.
    \]
    Following the proof above, we have 
    \begin{align}
        |hA|
        &\leq \frac{\sum_{p\in [y,2y]}\log p-\log (hN)}{\sum_{p\in [y,2y]}\frac{\log p}{|(hA)_p|}-\log (hN)}\leq \frac{2y}{\sum_{p\in [y,2y]}\frac{\log p}{hp^\theta}-\log (hN)} \notag\\
        &\leq \frac{2y}{\frac{2^{1-\theta}-1}{2h(1-\theta)}y^{1-\theta}-\log (hN)}\leq \frac{8h}{C}y^\theta \label{eq:Gsieve_cond3}, 
    \end{align}
    provided that \begin{equation*}
    \frac{C}{h}y^{1-\theta} > 4\log (hN).
    \end{equation*}
    Since $y> (16\log N)^{1/\eps}$, $C\geq 1/2$, we have
    $$
    \frac{C}{h}y^{1-\theta}\geq \frac{C^2}{40}\delta y^{1-\theta} \geq 2C^2y^{\eps}\geq 8\log N.
    $$
    Thus, when $N$ is sufficiently large compared to $\theta$ and $\eps$,   
    inequality \eqref{eq:Gsieve_cond3} holds. It follows that
    \[
    |hA|\leq \frac{8h}{C}y^\theta\leq \frac{h(h+1)}{3}\delta y^\theta\leq \frac{h(h+1)}{2}(|A|-2),
    \]
    and Lemma~\ref{lem:lev2} implies that there exists an arithmetic progression $Q$ such that $A\subseteq Q$ and 
    \[|Q|\leq h|A|\leq \frac{32}{C\delta}|A| \leq 64y^\theta. \qedhere\]

\section{Applications of the inverse theorems} 
\label{sec:app}
In this section, we discuss some applications of the inverse theorems we proved.

\subsection{Union of several arithmetic progressions: proof of Theorem~\ref{thm:kAP}}

\begin{proof}[Proof of Theorem~\ref{thm:kAP}]
Define
\begin{equation}\label{eq:Delta}
\Delta=\Delta(k,\eps)=\exp(3k\log k/\eps).
\end{equation}
Let $C$ be the absolute constant from Theorem~\ref{maintheorem}. For each subset of primes $\mathcal{P} \subseteq [y,2y]$, and each $1\leq i \leq k$, define
$$
S^{(i)}(\mathcal{P})=\{0\leq n \leq N: n \in R_p^{(i)} \pmod p \text{ for all } p \in \mathcal{P}\}.
$$
By Theorem~\ref{maintheorem}, if $|\mathcal{P}|\geq Cy^{1-2\eps} \log N /\log y$, then $S^{(i)}(\mathcal{P})$ is an arithmetic progression of length at most $(2y)^{1/2-\eps}+1$; thus, if $|S^{(i)}(\mathcal{P})|\geq 2$, we can define $D(S^{(i)}(\mathcal{P}))$ to be the step size of the arithmetic progression $S^{(i)}(\mathcal{P})$. Also observe that if $\mathcal{P} \supset \mathcal{P}'$ with $|\mathcal{P}'|\geq Cy^{1-2\eps} \log N /\log y$,  $|S^{(i)}(\mathcal{P})|\geq 2$, and $D(S^{(i)}(\mathcal{P}))>D(S^{(i)}(\mathcal{P}'))$, then
$$
|S^{(i)}(\mathcal{P}')|\geq \frac{D(S^{(i)}(\mathcal{P}))}{D(S^{(i)}(\mathcal{P}'))} (|S^{(i)}(\mathcal{P})|-1)+1>\frac{D(S^{(i)}(\mathcal{P}))}{2D(S^{(i)}(\mathcal{P}'))} |S^{(i)}(\mathcal{P})|.
$$

 The key ingredient of our proof is the following arithmetic regularity lemma.

\begin{claim}[Regularity lemma]\label{lem:RL}
There is a subset of primes $\mathcal{P} \subseteq [y,2y]$  with $$|\mathcal{P}|\geq \frac{Cky^{1-2\eps} \log N}{\log y}$$  such that for each $\mathcal{P}' \subseteq \mathcal{P}$ with $|\mathcal{P}'|\geq |\mathcal{P}|/k$ and each $1\leq i \leq k$, one of the following holds:
\begin{enumerate}
    \item If $|S^{(i)}(\mathcal{P})|\geq 2$, then $D(S^{(i)}(\mathcal{P}'))\geq D(S^{(i)}(\mathcal{P}))/\Delta$.
    \item If $|S^{(i)}(\mathcal{P})|\leq 1$, then $S^{(i)}(\mathcal{P}')=S^{(i)}(\mathcal{P})$.
\end{enumerate}
\end{claim}
\begin{poc}
We proceed with proof by contradiction. Suppose otherwise that such a subset $\mathcal{P}$ does not exist. Then we can find a chain of subsets of primes:
$$
\mathcal{P}_0 \supset \mathcal{P}_1 \supset \cdots \mathcal{P}_m
$$
satisfying the following properties:
\begin{itemize}
    \item $\mathcal{P}_0$ is the set of all primes in $[y,2y]$,
    \item  $|\mathcal{P}_{j+1}|\geq |\mathcal{P}_{j}|/k$ for each $0\leq j \leq m-1$,
    \item $Cy^{1-2\eps} \log N /\log y\leq |\mathcal{P}_m|<Cky^{1-2\eps} \log N /\log y$,
    \item For each $0\leq j \leq m-1$, there is some $1\leq i_j\leq k$, such that one of the following holds:
    \begin{enumerate}
        \item (Case 1) $|S^{(i_j)}(\mathcal{P}_j)|\geq 2$ and $D(S^{(i)}(\mathcal{P}_{j+1}))\leq D(S^{(i_j)}(\mathcal{P}_j))/\Delta$. In this case, we have
$
|S^{(i_j)}(\mathcal{P}_{j+1})|\geq \Delta|S^{(i_j)}(\mathcal{P}_j)|/2.
$
        
        \item (Case 2) $|S^{(i_j)}(\mathcal{P}_j)|<2$ and $|S^{(i_j)}(\mathcal{P}_{j+1})|>|S^{(i_j)}(\mathcal{P}_j)|$.
    \end{enumerate}
\end{itemize}

Note that the number of $j$ such that Case 2 occurs is at most $2k$. Thus, by the pigeonhole principle, we can find some $1\leq i^*\leq k$ such that Case 1 occurs for at least $(m-2k)/k$ many different $j$'s with $0\leq j \leq m-1$ and $i_j=i^*$. In particular, we have
$$
|S^{(i^*)}(\mathcal{P}_m)|\geq (\Delta/2)^{(m-2k)/k}.
$$
On the other hand, since $|\mathcal{P}_m|\geq Cy^{1-2\eps} \log N /\log y$, we know that $S^{(i^*)}(\mathcal{P}_m)$ is an arithmetic progression of length $(2y)^{1/2-\eps}+1$. Thus,
$$
(\Delta/2)^{(m-2k)/k} \leq |S^{(i^*)}(\mathcal{P}_m)| \leq (2y)^{1/2-\eps}+1.
$$
It follows that 
\begin{equation}\label{eq:lb}
\frac{m \log \Delta}{k} \leq (1+o(1)) \log y.
\end{equation}

Also, by the assumption on the sizes of $\mathcal{P}_{j}'s$, we have
$$
\frac{y}{\log y} \ll |\mathcal{P}_0|\leq k^m |\mathcal{P}_m|\ll \frac{k^{m+1}y^{1-2\eps}\log N}{\log y},
$$
and thus
$$y^{2\eps}\ll k^{m+1}\log N.$$
Since $y\geq (\log N)^{1/\eps}$, it follows that $\log N \leq y^{\eps}$. Thus,
\begin{equation}\label{eq:ub}
\eps \log y \leq (1+o(1))(m+1)\log k \leq (2+o(1))m\log k.
\end{equation}

Comparing inequalities~\eqref{eq:lb} and~\eqref{eq:ub},  we obtain that
$$
\log \Delta \leq \frac{(2+o(1))k\log k}{\eps},
$$
contradicting the choice of $\Delta$ in equation~\eqref{eq:Delta}. This completes the proof of the claim.
\end{poc}

Let $\mathcal{P}$ be from Claim~\ref{lem:RL}. By the pigeonhole principle, for each $n\in A$, we can find some $1\leq i \leq k$ such that $n \in R_p^{(i)} \pmod p$ for at least $|\mathcal{P}|/k$ many primes $p \in \mathcal{P}$. It follows that
$$
A\subseteq \bigcup_{i=1}^k \bigcup_{\substack{\mathcal{P}' \subseteq \mathcal{P}\\ |\mathcal{P}'|\geq |\mathcal{P}|/k}} S^{(i)}(\mathcal{P}'):=\bigcup_{i=1}^k T_i.
$$    
Thus, it remains to show that $T_i$ is contained in an arithmetic progression of length $\ll y^{1/2-\eps}$ for each $1\leq i \leq k$.

Let $\Delta'=\operatorname{lcm} \{1,2, \ldots, \Delta\}$. By the prime number theorem, $\Delta'=\exp(O(\Delta))$.

Fix $1\leq i \leq k$. If $|S^{(i)}(\mathcal{P})|\leq 1$, then $S^{(i)}(\mathcal{P}')=S^{(i)}(\mathcal{P})$ for all $\mathcal{P}' \subseteq \mathcal{P}$ with $|\mathcal{P}'|\geq |\mathcal{P}|/k$, and it follows that $|T_i|\leq 1$ and we are done. Next assume that $|S^{(i)}(\mathcal{P})|\geq 2$. Let $d=D(S^{(i)}(\mathcal{P}))$ and pick $a \in S^{(i)}(\mathcal{P})$. For each $\mathcal{P}' \subseteq \mathcal{P}$ with $|\mathcal{P}'|\geq |\mathcal{P}|/k$, we have $|\mathcal{P}'|\geq Cy^{1-2\eps} \log N /\log y$, $a \in S^{(i)}(\mathcal{P}')$, $D(S^{(i)}(\mathcal{P}')) \mid d$ with $D(S^{(i)}(\mathcal{P}'))\geq d/\Delta$, and thus $$S^{(i)}(\mathcal{P}') \subseteq \{a+jD(S^{(i)}(\mathcal{P}')) : -(2y)^{1/2-\eps}\leq j\leq (2y)^{1/2-\eps}\}.
$$
Let $d'$ be the greatest common divisor of $D(S^{(i)}(\mathcal{P}'))$ among all subsets $\mathcal{P}' \subseteq \mathcal{P}$ with $|\mathcal{P}'|\geq |\mathcal{P}|/k$. Then $d' \mid d$ and $\frac{d}{d'}\leq \Delta'$. It follows that 
$$
T_i \subseteq \{a+jd' : -(2y)^{1/2-\eps}\Delta'\leq j\leq (2y)^{1/2-\eps}\Delta'\},
$$
that is, $T_i$ is contained in an arithmetic progression of length at most $$2 (2y)^{1/2-\eps} \cdot \exp(O(\Delta))+1=\exp(\exp(O(k\log k/\eps))) y^{1/2-\eps},$$ as required. 
\end{proof}

\begin{rem}\label{rem:kAP}
In the statement of Theorem~\ref{thm:kAP}, we required $n\in R_p \pmod p$ for each prime $p\in [y,2y]$. By slightly modifying the above proof, we can show that if $\delta \in (0,\eps)$ and $n\in R_p \pmod p$ holds for at least $y^{1-\eps+\delta}$ primes $p\in [y,2y]$, then $A$ is contained in the union of the $k$ arithmetic progressions, each of length $\ll_{k,\eps,\delta} y^{1/2-\eps}$. This observation will be needed in the proof of Theorem~\ref{thm:kAP_sieve}.
\end{rem}

The next remark shows that Theorem~\ref{maintheorem} cannot be extended to the setting where $R_p$ is a union of $k$ short arithmetic progressions for each prime $p\in [y,2y]$.

\begin{rem}\label{rem:kAPno}
Let $y=N/2$ and $k\geq 2$. Choose $k$ positive integers $L_1,\ldots,L_k$ such that $L_{i+1}\geq 4L_i$ for all $1\leq i \leq k-1$, $\gcd(L_i,L_j)=1$ for all $1\leq i<j\leq k$, and $L_k<y^{1/k-\eps}$. For any $1\leq i\leq k$ and all the primes $p\in [(1+\frac{i-1}{k})y,(1+\frac{i}{k})y],$ let 
\[R_p = \bigg\{L_i,2L_i,\ldots,\bigg(\prod_{r\neq i}L_r\bigg)L_i\bigg\} \cup\bigcup_{\substack{1\leq j\leq k}}\{0,L_j,2L_j\}\subseteq\Z_p.\] 
Let $A$ be the set of all integers $n\in [0,N]$ such that  $n$ mod $p$ is in $R_p$ for every prime $p \in [y,2y]$. It is easy to verify that $A = \cup_{j=1}^k\{0,L_j,2L_j\}\cup\{\prod_{i=1}^kL_i\}$. Consequently, $A$ cannot be written as a union of $k$ arithmetic progressions because $L_i$'s grow exponentially.
\end{rem}

\subsection{Improved larger sieve}

\begin{proof}[Proof of Theorem~\ref{thm:interval_sieve}]
    Let $y=(16\log N)^{1/\eps}$ and $\mathcal{P}$ be the set of all primes in $[y,2y]$. It follows from Theorem~\ref{thm:interval} that there is an interval $Q = \{a,a+1,\ldots,a+\ell\}$ with $\ell\leq (2y)^{1-\eps}$, such that $A\subseteq Q$ and $Q\pmod p\subseteq I_p$ for every $p\in [y,2y]$. 

    Next we replace $N$ by $\ell$, $A$ by $A-a$, $y$ by $(16\log \ell)^{1/\eps}$ and repeat the above argument. This procedure terminates when we have either $\ell\leq N_0(\eps)$ or $(16\log \ell)^{1/\eps}\leq p_0$. If $\ell\leq N_0(\eps)$, then $A$ is contained in an interval of length at most $N_0(\eps)$; if $(16\log \ell)^{1/\eps}\leq p_0$, then we replace $y$ by $p_0$, $A$ by $A-a$, and $N$ by $\lceil\exp(p_0^{\eps}/16)\rceil$, and then apply Theorem~\ref{thm:interval} once again. It follows that $A$ is contained in an interval of length at most $p_0^{1-\eps}+1$. In conclusion, $A$ must be contained in an interval of length at most $\max(p_0^{1-\eps}+1,N_0(\eps))$ as required.
\end{proof}

Next, we use a similar idea to prove Theorem~\ref{thm:kAP_sieve}.

\begin{proof}[Proof of Theorem~\ref{thm:kAP_sieve}]
    Let $y=(\log N)^{1/\eps}/2$. It follows from Theorem~\ref{thm:kAP} that there exists $k$ arithmetic progressions $Q_1,\ldots ,Q_k$ with $|Q_i|\ll_{k,\eps} y^{1/2-\eps}$ for all $1\leq i\leq k$, such that $A\subseteq\cup_{i=1}^k Q_i$. To prove the theorem, it suffices to show $|A \cap Q_i| \ll_{k,\eps}(\log N)^{1/2-\eps}+p_0^{1/2-\eps}$ for each $1\leq i \leq k$.
        
    Now fix some $1\leq i\leq k$. If $|Q_i|\leq 1$, we are done. Next assume that $|Q_i|\geq 2$, say $Q_i = \{a_i,a_i+d_i,\ldots,a_i+\ell_id_i\}$. Let
    \[A_i = \frac{1}{d_i}((A\cap Q_i)-a_i),\quad \mathcal{P}_i = \{p\in [C'\log N,2C'\log N]: p\nmid d_i\},\]
    where $C'$ is a constant to be determined later.
    Then $A_i\subseteq [0,\ell_i]$ and for each prime $p\in \mathcal{P}_i$, $A_i\pmod p$ is contained in the union of $k$ arithmetic progressions $\bigcup_{j=1}^k \overline{d_i}(R_p^{(j)}-a_i) \subset \Z_p$, where $\overline{d_i}$ is the multiplicative inverse of $d_i$ modulo $p$. Since $d_i\leq N$, it has at most $\log N/\log (C'\log N)$ distinct prime factors in the interval $[C'\log N,2C'\log N]$. By the prime number theorem, we have 
    \[|\mathcal{P}_i|\geq \frac{C'\log N}{2\log (C'\log N)}-\frac{\log N}{\log(C'\log N)}.\] 
    Note that $\log \ell_i\ll_{k,\eps}\log \log N$, thus we may choose $C'=C'(k,\eps)$ to be large enough so that $C'\log N\geq (\log \ell_i)^{1/\eps}$ and $|\mathcal{P}_i| \geq (C'\log N)^{1-\eps/2}$.
    
    If $p_0\leq C'\log N$, then it follows from Theorem~\ref{thm:kAP} that $A_i$ is contained in the union of $k$ arithmetic progressions, each of length $\ll_{k,\eps}(\log N)^{1/2-\eps}$; if $p_0>C'\log N$, then we instead replace $\mathcal{P}_i$ by $\mathcal{P}_i' = \{p\in [p_0,2p_0]:p\nmid d_i\}$ and then apply Theorem~\ref{thm:kAP} to deduce that $A_i$ is contained in the union of $k$ arithmetic progressions, each of length $\ll_{k,\eps}p_0^{1/2-\eps}$. Thus, we have $|A_i|\ll_{k,\eps} (\log N)^{1/2-\eps}+p_0^{1/2-\eps}$, as required.
\end{proof}

We end the section by illustrating the sharpness of these two theorems.

\begin{rem}\label{rem:sieve}
Take $A = \{0,1,2,\ldots, \lfloor p_0^{1-\eps}\rfloor\}$. Then obviously $A_p \subseteq \{0,1,2,\ldots, \lfloor p^{1-\eps}\rfloor\}$ for each $p\geq p_0$. Thus, Theorem~\ref{thm:interval_sieve} is sharp when $p_0^{1-\eps}+1>N_0(\eps)$.

Our bound on $|A|$ in Theorem~\ref{thm:kAP_sieve} is optimal up to the implied constant. When $p_0\geq \log N$, we can take $A = \{0,1,2,\ldots, \lfloor p_0^{1/2-\eps}\rfloor\}$ so that $A_p$ is contained in an arithmetic progression of length at most $p^{1/2-\eps}$ for each $p\geq p_0$. When $p_0$ is small (for example $p_0=1$), let 
\[A = \bigg\{j\prod_{p \leq 0.5\log N}p: 0\leq j\leq (0.5\log N)^{1/2-\eps}\bigg\}\subset [0,N].\]
Then $A_p = \{0\}$ for $p\leq 0.5\log N$ and $A_p$ is an arithmetic progression of length at most $p^{1/2-\eps}$ for $p>0.5\log N$. 
\end{rem}

\section{Inverse results concerning Generalized Arithmetic Progressions}\label{sec:GAP}
In this section, we prove Theorems~\ref{thm:energy},~\ref{thm:inverseGAP},~\ref{thm:inverseGAPsamerank}, and~\ref{thm:largeGAP}. 

\subsection{Proof of Theorems~\ref{thm:energy}}
Recall that $E(S)=\#\{(a,b,c,d) \in S^4:a+b=c+d\}$.
    Since $S\subseteq[N]$, it follows that
    \[\prod_{\substack{a,b,c,d\in S\\a+b\neq c+d}}|a+b-c-d|\leq (2N)^{|S|^4-E(S)}.\]
    Taking the logarithm on both sides, we obtain that
    \begin{equation}\label{eq:log}
    \sum_{\substack{a,b,c,d\in S\\a+b\neq c+d}}\log |a+b-c-d| \ll |S|^4\log N.
    \end{equation}

    Next, we lower bound the left-hand side of inequality~\eqref{eq:log} as follows:
    \[\sum_{\substack{a,b,c,d\in S\\a+b\neq c+d}}\log |a+b-c-d| \geq \sum_{p\in [y,2y]}\sum_{\substack{a,b,c,d\in S\\a+b\neq c+d}}1_{p\mid a+b-c-d}\log p \geq \sum_{p\in [y,2y]}(E_p(S)-E(S))\log p,\]
    where $E_p(S) = \#\{(a,b,c,d)\in S^4:a+b\equiv c+d\pmod p\}.$ Since $E_p(S) \geq E(S_p)\geq \delta |S_p|^3$ holds for each $p\in [y,2y]$, it follows from the prime number theorem that 
    \[\sum_{p\in [y,2y]}(E_p(S)-E(S))\log p\gg \delta\sum_{p\in [y,2y]}|S_p|^3\log p\ -E(S)y.\]
    Therefore, inequality~\eqref{eq:log} implies that
    \begin{equation}\label{eq:S^4}
    |S|^4\log N\gg \delta \sum_{p\in [y,2y]}|S_p|^3\log p\ -E(S)y.
    \end{equation}

 On the other hand, from Lemma~\ref{GS} and the prime number theorem, we have
    \begin{equation}\label{eq:Gsieve}
        |S|\ll \frac{\underset{p\in [y,2y]}\sum\log p - \log N}{\underset{p \in [y,2y]}\sum\frac{\log p}{|S_p|}-\log N} \ll_{\theta} \frac{y}{\underset{p \in [y,2y]}\sum\frac{\log p}{|S_p|}-\log N}.
    \end{equation}
    By H{\"o}lder's inequality and the prime number theorem,
    \[\bigg(\sum_{p\in[y,2y]}|S_p|^3\log p\bigg)\bigg(\underset{p \in [y,2y]}\sum\frac{\log p}{|S_p|}\bigg)^3 \geq \bigg(\sum_{p\in[y,2y]}\log p\bigg)^4 \gg y^4.\]
    Hence
    \begin{equation}\label{eq:S^3}
        \sum_{p\in[y,2y]}|S_p|^3\log p \gg y^4\bigg(\sum_{p\in[y,2y]}\frac{\log p}{|S_p|}\bigg)^{-3}.
    \end{equation}
    Since $y\geq (\log N)^{2/(1-\theta)}$ and $|S_p|\leq p^\theta$ for each prime $p\in [y,2y]$, it follows that 
    \[\sum_{p\in[y,2y]}\frac{\log p}{|S_p|}\geq \sum_{p\in [y,2y]}\frac{\log p}{p^\theta} \gg_{\theta} y^{1-\theta}\gg (\log N)^2.\]
    Thus inequalities \eqref{eq:Gsieve} and \eqref{eq:S^3} imply that
    \begin{align*}
    |S|^4\log N
    \ll_{\theta} \frac{y^4\log N}{(\sum_{p\in[y,2y]}\frac{\log p}{|S_p|})^{4}}\ll \frac{\log N \cdot \sum_{p\in[y,2y]}|S_p|^3\log p}{\sum_{p\in[y,2y]}\frac{\log p}{|S_p|}}
    \ll_{\theta} \frac{\sum_{p\in[y,2y]}|S_p|^3\log p}{\log N}.
    \end{align*}
    Now inequalities~\eqref{eq:S^4}, \eqref{eq:Gsieve}, and \eqref{eq:S^3} together imply that \[E(S)\gg_{\theta} \frac{\delta}{y} \sum_{p\in[y,2y]}|S_p|^3\log p\gg_\theta \delta\bigg(\frac{y}{\sum_{p\in[y,2y]}\frac{\log p}{|S_p|}}\bigg)^3\gg_\theta \delta|S|^3,\] as required. 

\subsection{Proof of Theorem~\ref{thm:inverseGAP}}
   
By translating $S$ and $R_p$ we may assume that $0\in S\subseteq[0,N-1]$ and $0\in R_p$ for all primes $p\in [y,2y]$. Since $N$ is large compared to $r$ and $\eps$, by the given lower bound on $y$, we can assume that $y$ is large enough. Since $y\geq (\log N)^{2/\eps}\geq (\log N)^{2/(1-\theta)}$, the prime number theorem and Lemma~\ref{GS} imply that
     \begin{equation}\label{eq:GSbound0}
     |S|\leq \frac{\sum_{p\in [y,2y]}\log p-\log N}{\sum_{p\in [y,2y]}\frac{\log p}{|S_p|}-\log N}\leq \frac{\sum_{p\in [y,2y]}\log p-\log N}{\sum_{p\in [y,2y]}\frac{\log p}{p^\theta}-\log N}\leq  \frac{2(1-\theta)}{2^{1-\theta}-1}y^\theta. 
     \end{equation}
Let $n$ be the minimal positive integer such that 
\begin{equation}\label{eq:choose_n}
2^n>2^{r+1}(3n+1)^r\delta^{-1}.
\end{equation}
Then we have $n \asymp r\log r+\log\delta^{-1}$. Let $h=3n+1$. By inequality~\eqref{eq:GSbound0}, we have 
\begin{equation}
    \frac{\delta y}{(2h)^r|S|}\geq \frac{(2^{1-\theta}-1)\delta y^{1-\theta}}{2(1-\theta)(6n+2)^r}\gg \frac{\delta y^{1-\theta}}{(r\log r+\log\delta^{-1})^r}\gg \frac{y^\eps}{(r\log r+\log y)^r}.
\end{equation}
Since $y\geq (\log N)^{2/\eps}$, when $N$ is sufficiently large depending on $r$ and $\eps$, we can guarantee that 
\begin{equation}\label{eq:Gsieve_hcond}
    \frac{\delta y}{(2h)^r|S|}\geq 8\log (2hN).
\end{equation}
For each prime $p\in [y,2y]$, we have $h(S\cup -S)_p\subseteq h(R_p\cup-R_p)\subseteq h(R_p-R_p)$; since $R_p$ is a GAP of rank at most $r$, it follows that 
    \[
    |h(R_p\cup -R_p)|\leq |h(R_p-R_p)|\leq (2h)^{r}|R_p|\leq (2h)^{r}\delta^{-1}|S_p|\leq (2h)^{r}\delta^{-1}|S|.
    \] 
Thus, by the prime number theorem and Lemma~\ref{GS}, we have
    \begin{align*}
        |h(S\cup -S)|
        &\leq \frac{\sum_{p\in [y,2y]}\log p-\log (2hN)}{\sum_{p\in [y,2y]}\frac{\log p}{|h(S\cup-S)_p|}-\log (2hN)} \notag\\
        &\leq \frac{\sum_{p\in [y,2y]}\log p-\log (2hN)}{\sum_{p\in [y,2y]}\frac{\log p}{(2h)^{r}\delta^{-1}|S|}-\log (2hN)}\leq 2\cdot(2h)^r\delta^{-1}|S|, 
    \end{align*}
where we used inequality~\eqref{eq:Gsieve_hcond} in the last step. From inequality~\eqref{eq:choose_n}, we know
\[|(3n+1)(S\cup -S)|\leq 2^{r+1}\cdot (3n+1)^r\delta^{-1}|S\cup -S|< 2^n|S\cup -S|.\]
It follows from Lemma~\ref{lem:polygrowth} that $S\cup -S$ has relative polynomial growth of order $O(n) = O(r\log r+\log\delta^{-1})$. Now we can apply Corollary~\ref{cor:polygrowth} to the set $S\cup -S$ to conclude that there is a GAP $Q$ of rank at most $O((r+\log\delta^{-1})^{1+o(1)})$, such that $S\subseteq(S\cup -S)-(S\cup -S)\subseteq Q$ and $$|Q|\leq \exp(O((r+\log\delta^{-1})^{1+o(1)}))|S\cup -S|\leq \exp(O((r+\log\delta^{-1})^{1+o(1)}))|S|.$$ 

\subsection{Proof of Theorems~\ref{thm:inverseGAPsamerank} and~\ref{thm:largeGAP}}

We conclude the paper with a proof of Theorems~\ref{thm:inverseGAPsamerank} and~\ref{thm:largeGAP}. 

\begin{proof}[Proof of Theorem~\ref{thm:inverseGAPsamerank}]
Since $N$ is large compared to $r$ and $\eps$, by the given lower bound on $y$, we can assume that $y$ is large enough. Since $y>(8^{r+3}\log N)^{1/(r\eps)}\geq (64\log N)^{1/(1-\theta)}$, the prime number theorem and Lemma~\ref{GS} imply that
     \begin{equation}\label{eq:upperboundS}
     |S|\leq \frac{\sum_{p\in \mathcal{P}}\log p-\log N}{\sum_{p\in \mathcal{P}}\frac{\log p}{|S_p|}-\log N}\leq \frac{\sum_{p\in \mathcal{P}}\log p-\log N}{\sum_{p\in \mathcal{P}}\frac{\log p}{p^\theta}-\log N}\leq 16 y^\theta. 
     \end{equation}
     
 Let $k$ be the unique positive integer such that 
    \begin{equation}\label{eq:2^k}
    (2\delta^{-1})^{\frac{1}{1-\eta}} < 2^k \leq 2(2\delta^{-1})^{\frac{1}{1-\eta}}.
    \end{equation}
    By inequalities~\eqref{eq:upperboundS} and~\eqref{eq:2^k}, and the given lower bound on $\delta$ and $y$, we have
    \begin{equation}\label{eq:Gsieve_cond}
    \frac{\delta y}{2^{kr}|S|}\geq \frac{\delta^{1+\frac{r}{1-\eta}}y}{2^{r+\frac{r}{1-\eta}}|S|}\geq \frac{\delta^{1+\frac{r}{1-\eta}}y^{1-\theta}}{2^{r+\frac{r}{1-\eta}+4}} \geq \frac{y^{r\eps/(1-\eta)}}{2^{3r+4}}\geq 32\log N.   
    \end{equation}
    By inequality~\eqref{eq:2^k} and the given lower bound on $\eta,\delta$, and $y$, we have
    \begin{equation}\label{eq:kbound}
    \log (2^k)< 1+\frac{\log (2\delta^{-1})}{1-\eta}<1+2\log (2y)< 3\log N.
    \end{equation}
    
    For each prime $p\in \mathcal{P}$, we have $(2^kS)_p\subseteq 2^kR_p$; since $R_p$ is a GAP of rank at most $r$, it follows that 
    \[
    |2^kR_p|\leq 2^{kr}|R_p|\leq 2^{kr}\delta^{-1}|S_p|\leq 2^{kr}\delta^{-1}|S|.
    \] 
Thus, by the prime number theorem and Lemma~\ref{GS}, we have
    \begin{align}
        |2^kS|
        &\leq \frac{\sum_{p\in \mathcal{P}}\log p-\log (2^k N)}{\sum_{p\in \mathcal{P}}\frac{\log p}{|(2^kS)_p|}-\log (2^k N)}\leq \frac{\sum_{p\in \mathcal{P}}\log p-\log (2^k N)}{\sum_{p\in \mathcal{P}}\frac{\log p}{2^{kr}\delta^{-1}|S|}-\log (2^k N)}\leq 2^{kr+1}\delta^{-1}|S|, \label{eq:Gsieve2j}
    \end{align}
    where we used $|\mathcal{P}|\geq y/(4\log y)$ and inequalities~\eqref{eq:Gsieve_cond} and~\eqref{eq:kbound} in the last step.  
    
   Now it follows from inequality~\eqref{eq:Gsieve2j} that
    \[
    \prod_{j=1}^k\frac{|2^jS|}{|2^{j-1}S|} = \frac{|2^kS|}{|S|} \leq 2^{kr+1}\delta^{-1}.
    \]
    By inequality~\eqref{eq:2^k} and the pigeonhole principle, there exists some integer $1\leq j_0\leq k$ such that 
    \[
    \frac{|2^{j_0}S|}{|2^{j_0-1}S|} \leq (2^{kr+1}\delta^{-1})^{1/k}<(2^{kr+1} \cdot 2^{k(1-\eta)}/2)^{1/k}=2^{r+1-\eta}.
    \]
    By Theorem~\ref{thm:FB}, there is an absolute constant $C>0$ so that we can cover $2^{j_0-1}S$ by at most 
    \[\exp(\exp(Cr))/(\eta/2)^{C2^r}\leq \exp(C\exp(Cr)\cdot\log\eta^{-1})\] translates of a GAP $Q$ with rank at most $r$ and 
    \[
    |Q|\leq |2^{j_0-1}S|\leq |2^kS|\leq 2^{kr+1}\delta^{-1}|S|\leq 2^{3r+1}\delta^{-(1+\frac{r}{1-\eta})}|S|
    \]
    by inequalities~\eqref{eq:Gsieve2j} and~\eqref{eq:2^k}. In particular, we can cover $S$ by at most $\exp(C\exp(Cr)\cdot\log \eta^{-1})$ translates of $Q$.
\end{proof}

Finally, we use Theorem~\ref{thm:inverseGAPsamerank} to deduce Theorem~\ref{thm:largeGAP}.

\begin{proof}[Proof of Theorem~\ref{thm:largeGAP}]
First, we choose $\eta \in (0,1/2)$ such that
$$
\frac{(\theta-1)(1-\eta)}{r+1-\eta}+\frac{2\eps}{3}=\frac{\theta-1}{r+1}+\eps.
$$
Note that
$$
\frac{\eps}{3}=\frac{(\theta-1)(1-\eta)}{r+1-\eta}-\frac{\theta-1}{r+1}=\frac{(1-\theta)\eta r}{(r+1)(r+1-\eta)} \asymp \frac{(1-\theta)\eta}{r},
$$
thus $\log \eta^{-1}\ll \log \eps^{-1}$.
We shall apply Theorem~\ref{thm:inverseGAPsamerank} with
$$
\delta=y^{-\frac{1-\theta}{r+1}+\eps}=y^{\frac{(\theta-1)(1-\eta)}{r+1-\eta}+\frac{2\eps}{3}}.
$$

Let $\mathcal{P}=\{p \in [y,2y]: |S_p|\geq \frac{\delta}{16} p^\theta\}$. We claim that $|\mathcal{P}|\geq y/(4\log y)$. Since $N$ is large enough, it follows from the lower bound on $y$ that $y$ is large enough. In particular, by the prime number theorem, 
\begin{equation}\label{eq:logN}
\underset{p \in [y,2y]}\sum\frac{\log p}{|S_p|}\geq \underset{p \in [y,2y]}\sum\frac{\log p}{p^\theta} \geq \frac{2^{1-\theta}-1}{2(1-\theta)} y^{1-\theta} \geq 2\log N.
\end{equation}
By Lemma~\ref{GS}, the prime number theorem, and inequality~\eqref{eq:logN}, we have
$$
\delta y^{\theta} \leq
|S|\leq \frac{\underset{p\in [y,2y]}\sum\log p - \log N}{\underset{p \in [y,2y]}\sum\frac{\log p}{|S_p|}-\log N} \leq \frac{4y}{\underset{p \in [y,2y]}\sum\frac{\log p}{|S_p|}},
$$
It follows that
\begin{equation}\label{eq:uba}
\underset{p \in [y,2y]}\sum\frac{\log p}{|S_p|} \leq 4 \delta^{-1} y^{1-\theta}.
\end{equation}
On the other hand, we have 
\begin{equation}\label{eq:lba}
\underset{p \in [y,2y]}\sum\frac{\log p}{|S_p|}
\geq \underset{p \in [y,2y]\setminus \mathcal{P}}\sum\frac{\log p}{|S_p|}\geq \underset{p \in [y,2y]\setminus \mathcal{P}}\sum\frac{8\log y}{\delta y^\theta}=8\frac{\#\{p: p\in [y,2y] \setminus \mathcal{P}\} \cdot \log y}{\delta y^\theta}.
\end{equation}
Since $y$ is large enough, we have $\#\{p: p\in [y,2y]\}\geq 3y/4\log y$. Thus, comparing inequalities~\eqref{eq:uba} and~\eqref{eq:lba}, it follows that $|\mathcal{P}|\geq y/(4\log y)$, as required.

It then follows from Theorem~\ref{thm:inverseGAPsamerank} that there exists a GAP $Q$ of rank at most $r$ and size at most $$
2^{3r+1}\delta^{-(1+\frac{r}{1-\eta})}|S|=2^{3r+1} y^{1-\theta-\frac{2\eps(1-\eta+r)}{3(1-\eta)}}|S|\leq 2^{3r+5} y,
$$
where we used inequality~\eqref{eq:upperboundS} in the least step; moreover, $S$ is covered by at most $\exp(C\exp(Cr)\cdot\log\eps^{-1})$ translates of $Q$, where $C$ is an absolute constant.
\end{proof}

\section*{Acknowledgments}
The third author was supported in part by an NSERC fellowship.

\end{document}